\documentclass[oneside,openany,a4paper]{amsart}
\usepackage{amssymb}
\usepackage[centertags]{amsmath}
\usepackage{amsthm}
\usepackage{amsfonts}
\usepackage{eucal}
\usepackage{mathrsfs}

\usepackage[all,cmtip]{xy}
\usepackage{graphicx, eepic}

\usepackage{lmodern}
\usepackage[utf8]{inputenc}
\usepackage[T1]{fontenc}

\usepackage{bbm}
\usepackage{pdfpages}
\usepackage{titletoc}
\usepackage{booktabs}
\usepackage[shortlabels]{enumitem}
\usepackage{mathtools}
\usepackage{thmtools}

\usepackage{epstopdf}
\usepackage{epsfig}
\usepackage{microtype}

\usepackage{hyperref}
\hypersetup{colorlinks}
\hypersetup{
    bookmarks=true,         
    unicode=false,          
    pdftoolbar=true,        
    pdfmenubar=true,        
    pdffitwindow=false,     
    pdfstartview={FitH},    
    pdftitle={My title},    
    pdfauthor={Author},     
    pdfsubject={Subject},   
    pdfcreator={Creator},   
    pdfproducer={Producer}, 
    pdfkeywords={keyword1} {key2} {key3}, 
    pdfnewwindow=true,      
    colorlinks=true,       
    linkcolor=red,          
    citecolor=magenta,        
    filecolor=violet,      
    urlcolor=blue      
}
\usepackage{multirow, url}
\usepackage{textcomp}

\DeclareMathAlphabet{\cmcal}{OMS}{cmsy}{m}{n}

\usepackage[hang,flushmargin]{footmisc} 

\newtheoremstyle{thm}
  {3pt}
  {3pt}
  {\em}
  {0pt}
  {\bfseries}
  {}
  {5pt}
  {}
\newtheoremstyle{rem}
  {3pt}
  {3pt}
  {}
  {0pt}
  {\bfseries}
  {}
  {5pt}
  {}
  
\newtheorem{thm}{Theorem}[section]
\newtheorem{cor}[thm]{Corollary}
\newtheorem{lem}[thm]{Lemma}
\newtheorem{prop}[thm]{Proposition}

\newtheorem{assu}[thm]{Assumption}

\theoremstyle{definition}

\theoremstyle{remark}
\newtheorem{rem}[thm]{{Remark}}

\newtheorem{exam}[thm]{\bf{Example}}

\numberwithin{equation}{section} \numberwithin{table}{section}
	
\newtheorem*{thm*}{\bf{Theorem}}
\newtheorem*{rem*}{Remark}
\newtheorem*{rems*}{Remarks}
\newtheorem*{exam*}{Example}
\newtheorem*{exams*}{Examples}

\newcommand{\F}{{\mathbb{F}}}

\newcommand{\Q}{{\mathbb{Q}}}

\newcommand{\bX}{{\mathbb{X}}}

\newcommand{\Z}{{\mathbb{Z}}}

\newcommand{\fb}{{\mathfrak{b}}}

\newcommand{\fp}{{\mathfrak{p}}}

\newcommand{\fP}{{\mathfrak{P}}}

\newcommand{\cE}{{\mathcal{E}}}

\newcommand{\cK}{{\mathcal{K}}}

\newcommand{\cM}{{\mathcal{M}}}

\newcommand{\cO}{{\cmcal{O}}}

\newcommand{\cQ}{{\mathcal{Q}}}

\def\a{\alpha}
\def\b{\beta}

\def\d{\delta}
\def\e{\epsilon}

\def\g{\gamma}
\def\k{\kappa}
\def\z{\zeta}
\def\p{\varphi}
\def\vp{\varpi}

\newcommand{\Gm}{{\mathbb{G}}_{\operatorname{m}}}

\newcommand{\zmod}[1]{{\Z/{#1}\Z}}

\newcommand{\surj}{\twoheadrightarrow}
\newcommand{\arinj}{\ar@{^(->}}
\newcommand{\arsurj}{\ar@{->>}}
\newcommand{\arsub}{\ar@{}[r]|-*[@]{\subset}}
\newcommand{\arsup}{\ar@{}[r]|-*[@]{\supset}}
\newcommand{\arcap}{\ar@{}[d]|-*[@]{\subset}}
\newcommand{\arcup}{\ar@{}[u]|-*[@]{\subset}}
\newcommand{\arin}{\ar@{}[u]|-*[@]{\in}}

\renewcommand{\pmod}[1]{{\,(\operatorname{mod}\hspace*{0.7mm} {#1})}}
\renewcommand{\mod}[1]{{\,\operatorname{mod}\hspace*{0.7mm} {#1}}}

\newcommand{\Ker}{\operatorname{Ker}}

\newcommand{\Ext}{{\operatorname{Ext}}}
\newcommand{\Hom}{{\operatorname{Hom}}}
\newcommand{\Gal}{{\operatorname{Gal}}}

\newcommand{\Aut}{{\operatorname{Aut}}}
\newcommand{\GL}{{\operatorname{GL}}}

\newcommand{\SL}{{\operatorname{SL}}}

\newcommand{\Spec}{{\operatorname{Spec}}}

\newcommand{\inv}{{\operatorname{inv}}}
\newcommand{\ord}{{\operatorname{ord}}}

\newcommand{\Frob}{{\operatorname{Frob}}}

\renewcommand{~}{\hspace{0.3mm}}

\newcommand{\rra}{\longrightarrow}
\newcommand{\wt}{\widetilde}

\newcommand{\ov}{\overline}

\mathchardef\hyp="2D

\newcommand{\hs}[1]{\hspace{#1 mm}}
\newcommand{\h}{\hspace{5mm}}
\newcommand{\vv}{\vspace{5mm}}
\newcommand{\xyv}[1]{\xymatrixrowsep{#1 pc}}
\newcommand{\xyh}[1]{\xymatrixcolsep{#1 pc}}

\newcommand{\qa}{{\quad \text{and} \quad}}

\newcommand{\abs}[1]{| #1 |}

\newcommand{\leg}[2]{\genfrac(){}{}{#1}{#2}}

\newcommand{\bbM}{\mathbb{M}}

\newcommand{\mcG}{\mathcal{G}}
\newcommand{\mcC}{\mathcal{C}}
\newcommand{\mcD}{\mathcal{D}}

\newcommand{\bfg}{\mathbf{g}}
\newcommand{\bfa}{\mathbf{a}}

\newcommand{\bfx}{\mathbf{x}}
\newcommand{\uHom}{\underline{\Hom}}

\newcommand{\Ob}{\operatorname{Ob}}
\newcommand{\id}{\operatorname{id}}
\newcommand{\oH}{\operatorname{H}}
\newcommand{\oC}{\operatorname{C}}

\newcommand{\act}[2]{{}^{#1}\!{#2}}

\newcommand{\ur}{\mathrm{ur}}

\newcommand{\ad}{\operatorname{Ad}}

\def\vvdots{\setbox0=\hbox{$\vdots$}\ht0=8pt\box0}

\let\:=\colon

\def\Un{ \mathrm{Un} }

\def\cQ{ {\mathcal Q} }

\def\bX{ \bar{X}}
\def\Ext{{ \operatorname{Ext} }}

\def\Spec{{ \operatorname{Spec} }}

\def\a{{ \alpha }}
\def\b{{ \beta} }
\def\g{{ \gamma }}
\def\d{{ \delta }}

\def\e{{ \epsilon }}
\def\F{ {\mathbb F} }

\def\Un{ \mathrm{Un} }

\def\z{\zeta}

\def\bs{ \backslash}

\def\G{{ \Gamma }}

\def\cE{ {\mathcal E}}

\def\Z{{ \mathbb{Z}}}

\def\bq{\begin{quote}}
\def\eq{\end{quote}}
\def\Aut{ \operatorname{Aut}}

\def\Q{\mathbb{Q}}

\def\invlim{\varprojlim}

\def\P{ \mathbb{ P}}

\def\be{\begin{equation}}
\def\ee{\end{equation}}
\def\k{ \kappa}

\def\b{ \bar}

\def\L{\Lambda}

\def\a{\alpha}

\def\b{\beta}

\def\bs{\begin{slide}}
\def\es{\end{slide} }

\def\Gm{\mathbb{G}_m}

\numberwithin{equation}{section}

\def\ord{\mathrm{ord}}

\def\ad{\mathrm{ad}}

\def\beq{\begin{equation}}
\def\eeq{\end{equation}}
\def\O{\cmcal{O}}

\def\fP{\mathfrak{P}}

\usepackage{graphicx}

\def\Ad{\mathrm{Ad}}

\def\ext2et{\Ext^2_{\Un(\bX)}}

\def\ms{\medskip}

\def\cK{{\mathcal K}}

\def\<{\langle}
\def\>{\rangle}
\def\Zn{\Z/n\Z}
\def\xx{{\times }}
\def\msloc{\mathcal{M}_S^{loc}}
\def\nZ{\frac{1}{n}\Z/\Z}

\def\Ainfty{A^{\infty}}

\title{Arithmetic Chern-Simons Theory II}

\author{Hee-Joong Chung}
\address{ H.J.C: Korea Institute for Advanced Study,
85 Hoegiro, Dongdaemun-gu,
Seoul 02455,
Republic of Korea}
\author{Dohyeong Kim}
\address{D.K: Department of
Mathematics, University of Michigan,
2074 East Hall,
530 Church Street,
Ann Arbor, MI 48109-1043 , U.S.A. }
\author{Minhyong Kim}
\address{M.K: Mathematical Institute, University of Oxford,
Woodstock Road, Oxford OX2 6GG, UK, and Korea Institute for Advanced Study,
85 Hoegiro, Dongdaemun-gu,
Seoul 02455,
Republic of Korea }
\author{Jeehoon Park}
\address{J.P: Department of Mathematics, 
Pohang University of Science and Technology,
77 Cheongam-ro, Nam-gu, Pohang, Gyeongbuk, Republic of Korea 37673 }
\author{Hwajong Yoo}
\address{H.Y: IBS Center for Geometry and Physics,
Mathematical Science Building, Room 108,
Pohang University of Science and Technology,
77 Cheongam-ro, Nam-gu, Pohang, Gyeongbuk, Republic of Korea 37673 }
\subjclass[2000]{Primary 11R04, 11R23, 11R37 ; Secondary 81T45}
\begin{document}                                                                          
\maketitle
\begin{flushright}
with Appendix \ref{sec:appendix B} by Behrang Noohi
\end{flushright}
\begin{abstract}
In this paper, we apply ideas of Dijkgraaf and Witten \cite{witten, DW} on $3$ dimensional topological quantum field theory to arithmetic curves, that is, the spectra of rings of integers in algebraic number fields. In the first three sections, we define  classical Chern-Simons actions on spaces of Galois representations. In the subsequent sections, we give formulas for computation in a small class of cases  and point towards some  arithmetic applications.
\end{abstract}

\setcounter{tocdepth}{1}
\tableofcontents
\section{The arithmetic Chern-Simons action: introduction and definition}\label{sec:classical case}
The purpose of this paper is to cast in concrete mathematical form the ideas presented in the preprint \cite{kim5}. The reader is referred to that paper for motivation and speculation. Since there is no plan to submit it for separate publication, we repeat here the basic constructions before going  on to a family of examples. 
This paper adheres, however, to a rather strict mathematical presentation. As we remind the reader below, the analogies in the background have come to be somewhat well-known under the heading of `arithmetic topology'. The emphasis of this paper, however, will be less on analogies, and more on  the possibility that specific technical tools of topology and physics can be imported into number theory.
\ms

Let $X=\Spec(\cmcal{O}_F)$, the spectrum of the ring of integers in a number field $F$. We  assume that $F$ is totally imaginary. Denote by $\Gm$ the \'etale sheaf that associates to a scheme the units in the global sections of its coordinate ring. We have the following canonical isomorphism (\cite[p. 538]{mazur}):
\[\label{eqn:*}
\inv: H^3(X, \Gm)\simeq \Q/\Z. \tag{$*$}
\]
This map is deduced from the `invariant' map of local class field theory. We will therefore use the same name for a  range of isomorphisms having the same essential nature, for example,
\[\label{eqn:**}
\inv:H^3(X, \Z_p(1))\simeq \Z_p, \tag{$**$}
\]
where $\Z_p(1)=\invlim_i \mu_{p^i}$, and  $\mu_n\subset \Gm$ is the sheaf of $n$-th roots of 1. This follows from the exact sequence
\[
0\to \mu_n \to \Gm
\stackrel{~(\cdot)^n}{\to}\Gm \to \Gm/(\Gm)^n \to 0.
\]
That is, according to \textit{loc. cit.}, 
\[
H^2(X,\Gm)=0,
\] while by \textit{op. cit.}, p. 551, we have 
$$
H^i(X,\Gm/(\Gm)^n)=0
$$ 
for $i\geq 1$. If we break up the above into two short exact sequences,
$$
0\to \mu_n\to \Gm\stackrel{(\cdot)^n}{\to}\cK_n\to 0,
$$
and
$$
0\to \cK_n \to \Gm\to \Gm/(\Gm)^n\to 0,
$$
we deduce 
$$
H^2(X, \cK_n)=0,
$$
from which it follows that
$$
H^3(X, \mu_n)\simeq \nZ,
$$
the $n$-torsion inside $\Q/\Z$. Taking the inverse limit over $n=p^i$ gives the second isomorphism above.  
The pro-sheaf $\Z_p(1)$ is a very familiar coefficient system for \'etale cohomology and (\ref{eqn:**}) is reminiscent of the fundamental class of a compact oriented three manifold for singular cohomology. Such an analogy was noted by Mazur around 50 years ago \cite{mazur2} and has been developed rather systematically by a number of mathematicians, notably, Masanori Morishita \cite{morishita}. Within this circle of ideas is included the analogy between knots and primes, whereby the map
\[
\Spec(\cmcal{O}_F/\mathfrak{P}_v)\rightarrowtail X
\]
from the residue field of a prime $\mathfrak{P}_v$ should be similar to the inclusion of a knot.  Let $F_v$ be the completion of $F$ at the prime $v$ and $\cmcal{O}_{F_v}$ its valuation ring. If one takes this analogy seriously (as did Morishita), the map 
$$
\Spec(\cmcal{O}_{F_v})\to X,
$$ 
should be similar to the inclusion of a handle-body around the knot, whereas 
$$
\Spec(F_v)\to X
$$ 
resembles the inclusion of its boundary torus\footnote{It is not clear to us that the topology of the boundary should really be a torus. This is reasonable if one thinks of the ambient space as a three-manifold. On the other hand, perhaps it's possible to have a notion of a knot in a {\em homology three-manifold} that has an exotic tubular neighbourhood?}. Given a finite set $S$ of primes, we consider the scheme
\[
X_S:=\Spec(\cmcal{O}_F[1/S])=X\setminus \{\mathfrak{P}_v\}_{v\in S}.
\]
Since a link complement is homotopic to the complement of a tubular neighbourhood, the analogy is then forced on us between $X_S$ and a  three manifold with boundary given by a union of tori, one for each `knot' in $S$. These of course are basic morphisms in $3$ dimensional topological quantum field theory \cite{atiyah}. From this perspective, perhaps the coefficient system $\Gm$ of the first isomorphism should have reminded us of the $S^1$-coefficient important in Chern-Simons theory \cite{witten, DW}. A more direct analogue of $\Gm$ is the sheaf $\cmcal{O}_M^\xx$ of invertible analytic functions on a complex variety $M$. However, for compact K\"ahler manifolds, the comparison isomorphism 
$$
H^1(M, S^1)\simeq H^1(M, \cmcal{O}_M^\xx)_0,
$$
where the subscript refers to the line bundles with trivial topological Chern class, is a consequence of Hodge theory. This indicates that in the \'etale setting with no natural constant sheaf of $S^1$'s, the familiar $\Gm$ has a topological nature, and can be regarded as a substitute\footnote{Recall, however, that it is of significance in Chern-Simons theory that one side of this isomorphism is purely topological while the other has an analytic structure.}. One problem, however, is that  the $\Gm$-coefficient computed directly gives divisible torsion cohomology, whence the need for considering coefficients like $\Z_p(1)$ in order to get functions of geometric objects having an analytic nature as arise, for example, in the theory of torsors for motivic fundamental groups \cite{CK, kim1, kim2, kim3, kim4}.
\ms

We now move to the  definition of the arithmetic Chern-Simons action. Let
\[\pi=\pi_1(X, \fb),
\]  
be the profinite \'etale fundamental group of $X$,
where we take 
$$
\fb: \Spec(\ov F)\to X
$$ 
to be the geometric point coming from an algebraic closure of $F$. 
Assume now that the group $\mu_n({\ov F})$ of $n$-th roots of $1$ is in $F$ and fix a trivialisation 
$\zeta_n: \zmod n \simeq \mu_n$. 
This induces the isomorphism
$$
\inv: H^3(X, \zmod n) \simeq H^3(X, \mu_n) \simeq \frac{1}{n}\Z/\Z.
$$
Now let $A$ be a finite group and fix a class $c\in H^3(A, \zmod n)$.
Let 
$$
\mathcal{M}(A):= \Hom_{cont}(\pi, A)/A
$$
be the set of isomorphism classes of principal $A$-bundles over $X$. Here, the subscript refers to continuous homomorphisms, on which $A$ is acting by conjugation. For $[\rho]\in \mathcal{M}(A)$, we get a class
$$
\rho^*(c)\in H^3(\pi, \Z/n\Z)
$$ 
that depends only on the isomorphism class $[\rho]$. Denoting by $\inv$ also the composed map 
\[
\xymatrix{
H^3(\pi,  \zmod n) \ar[r] & H^3(X, \zmod n) \ar[r]^-{\inv}_-{\simeq} & \frac{1}{n}\Z/\Z.
}
\]
We get thereby a function
\[
\xyv{0.15}
\xymatrix{
CS_c: \mathcal{M}(A) \ar[r] & \nZ;\\
\hspace{12mm} [\rho] \hspace{2mm}\ar@{|->}[r] & \inv(\rho^*(c)).
}
\]
This is the basic and easy case of the classical Chern-Simons action\footnote{The authors realise that this terminology is likely to be unfamiliar, and maybe even appears pretentious to number-theorists. However, it does seem to encourage the reasonable view that  concepts and structures from geometry and physics can be specifically useful in number theory.} in the arithmetic setting.
\ms

Sections \ref{sec:with boundaries}  sets down some definitions for `manifolds with boundary', that is, $X_S$ as above. In fact, it turns out that the Chern-Simons action with boundaries is necessary for the computation of the action even in the `compact' case, in a manner strongly reminiscent of computations in topology (see \cite[Theorem 1.7 (d)]{FQ}, for example).  That is, we will compute the Chern-Simons invariant of a representation $\rho$ of $\pi$ using a suitable decomposition
$$X``="X_S\cup [\cup_v \Spec(\cO_{F_v})]$$
and restrictions of $\pi $ to $X_S$ and the $\Spec(\cO_{F_v})$.

\ms To describe the construction, we need more notations.
We assume that all primes of $F$ dividing $n$ are in the finite set of primes $S$.
Let 
$$
\pi_S:=\pi_1(X_S, \fb)
$$ 
and 
$$
\pi_v=\Gal({\ov F}_v/F_v) 
$$ 
equipped with maps
$$
i_v: \pi_v\to \pi_S
$$
given by choices of embeddings ${\ov F}\rightarrowtail {\ov F}_v$.  The collection 
$$
\{i_v\}_{v\in S}
$$ 
will be denoted by $i_S$.
There is a natural quotient map 
$$
\k_S:\pi_S \to \pi.
$$
Let
$$
Y_S(A):=\Hom_{cont}(\pi_S, A)
$$ 
and denote by
$\mathcal{M}_S(A)$ the action groupoid whose objects are the elements of $Y_S(A)$ with morphisms given by the conjugation action of $A$. We also have the local version
$$
Y_S^{loc}(A) :=\prod_{v\in S} \Hom_{cont}(\pi_v, A)
$$
as well as the action groupoid $\mathcal{M}_S^{loc}(A)$  with objects 
$Y_S^{loc}(A)$ and morphisms given by the action of $A^S:=\prod_{v\in S} A$
conjugating the separate components in the obvious sense.
 Thus, we  have the restriction functor 
 $$
 r_S: \mathcal{M}_S(A)\to \mathcal{M}_S^{loc}(A),
 $$
 where a homomorphism $\rho: \pi_S\to A$ is restricted to the collection
 $$
 r_S(\rho)=i_S^*\rho:=(\rho\circ i_v)_{v\in S}.
 $$ 
\ms

We will construct, in Section \ref{sec:with boundaries}, a functor $L$ from $\msloc(A)$ to the $\nZ$-torsors as a finite arithmetic version of the Chern-Simons line bundle \cite{FQ}  over $\mathcal{M}_S^{loc}(A)$. To a global representation $\rho\in \mathcal{M}_S(A)$, the Chern-Simons action will then associate an element (\ref{CSA with boundary}) 
$$CS_c([\rho]) \in L(r_S(\rho)).$$
\ms

Now, given
 $[\rho] \in \cM(A)$, we pull it back to $[\rho \circ \k_S] \in \cM_S(A)$ and apply the Chern-Simons action with boundary  to get an element 
\[
CS_c([\rho \circ \k_S]) \in L([r_S(\rho \circ \k_S) ]).
\]
On the other hand, for each $v \in S$, we can pull back $\rho$ to a local unramified representation
\[
\rho_v^{\ur}: \pi_v^{\ur} \to \pi \to A,
\]
where $\pi_v^{\ur}$ is the unramified quotient of $\pi_v$. The extra structure of the unramified representation will then allow us to canonically associate an element
\[
\sum_{v\in S} (\beta_v) \in L([r_S(\rho \circ \k_S)]),
\]
which can be interpreted as the Chern-Simons action of $(\rho_v^{\ur})_{v\in S}$ on $\cup_{v\in S} \Spec(\cO_{F_v})$.

\begin{thm}[The Decomposition Formula]\label{main theorem}
 Let $A$ be a finite group and fix a class $c\in H^3(A, \zmod n)$. Then
\[
CS_{c} ([\rho])= \sum\limits_{v \in S}(\beta_v) -CS_c([\rho \circ \k_S])
\]
for $[\rho] \in \cM(A)$.
\end{thm}

Section \ref{sec: towards computation} is devoted to a proof of Theorem \ref{main theorem}.
The key point of this formula is that $CS_c([\rho])$ can be computed as the difference between two trivialisations of the torsor, a ramified global trivialisation and an unramified local trivialisation.

\ms
In Section \ref{sec: examples}, we use this theorem to compute the Chern-Simons action for a class of examples. It is amusing to note the form of the action when $A$ is finite cyclic. That is,  let $A=\zmod n$, $\a\in H^1(A,\Z/n\Z)$ the class of the identity, and $\b\in H^2(A, \Z/n\Z)$ the class of the extension
$$
\xymatrix{
0 \ar[r] & \zmod n \ar[r]^-n & \zmod {n^2} \ar[r] & A \ar[r] & 0.
}
$$
Then $\b=\d\a$, where $\d:H^1(A, \Zn) = H^1(A, A) \to H^2(A,\Zn)$ is the boundary map arising from the extension. Put 
$$
c:=\a\cup \b=\a\cup \d \a\in H^3(A, \zmod n).
$$
Then
$$
CS_c([\rho])=\inv[ \rho^*(\a) \cup \d \rho^*(\a)],
$$
in close analogy to the formulas of abelian Chern-Simons theory.
\ms

However, our computations are not limited to the case where $A$ is an abelian cyclic group. 
Along  similar lines, we will provide an infinite family of number fields $F$  and representations $\rho$ such that  $CS_c([\rho])$ is non-vanishing for $[\rho]\in \cM(A)$ with a different class $c \in H^3(A,\zmod 2)$ and both abelian $A$ (see Propositions \ref{abelian1}, \ref{prop:V4 general example}, and \ref{prop:V4 example3}) and non-abelian $A$ (see Proposition \ref{nonabelian}).
\ms

In Section \ref{sec: application}, we provide arithmetic applications to a class of Galois embedding problems using the fact that the existence of an unramified extension forces a Chern-Simons invariant to be zero. 

\ms
 In this paper, we do not develop a $p$-adic theory in the case where the boundary is empty. In future papers, we hope to apply  local trivialisations using Selmer complexes to remedy this omission and complete the theory begun in Section \ref{sec:p-adic case}. To get actual $p$-adic functions, one needs of course to come to an understanding of explicit cohomology classes on $p$-adic Lie groups,  possibly by way of the theory of Lazard \cite{lazard}.
 Suitable quantisations of the theory of this paper in a manner amenable to arithmetic applications will be explored as well in future work, as in  \cite{CKKPPY},  where a precise arithmetic analogue of a `path-integral formula' for arithmetic linking numbers is proved. In that preprint, a  connection is made also to the class invariant homomorphism from additive Galois module structure theory. A pro-$p$ version of this homomorphism
is related to $p$-adic $L$-functions and heights, providing some evidence for the speculation from \cite{kim5}.

\section{The arithmetic Chern-Simons action: boundaries}\label{sec:with boundaries}


We keep the notations as in the introduction. 
We will now employ a cocycle $c\in Z^3(A, {\zmod n})$ to associate a $\nZ$-torsor  to each point of $Y_S^{loc}(A)$ in an $A^S$-equivariant manner. 
We use the notation
$$
C^i_S:=\prod_{v\in S} C^i(\pi_v, {\zmod n})
$$
for the continuous cochains,
$$
Z^i_S:=\prod_{v\in S} Z^i(\pi_v, {\zmod n})\subset C^i_S
$$
for the cocycles, and
$$
B^i_S:=\prod_{v\in S} B^i(\pi_v, {\zmod n})\subset Z^i_S\subset C^i_S
$$
for the coboundaries.
In particular, we have the coboundary map (see Appendix \ref{sec:Appendix A} for the sign convention)
$$
d:C^2_S\to Z^3_S.
$$

Let $\rho_S:=(\rho_v)_{v\in S}\in Y_S^{loc}(A)$ and put
$$
c\circ \rho_S:=(c\circ \rho_v)_{v\in S},
$$ 
$$
c\circ \Ad_a:=(c\circ \Ad_{a_v})_{v\in S}
$$
 for $a=(a_v)_{v\in S}\in A^S$, where $\Ad_{a_v}$ refers to the conjugation action.  To define the arithmetic Chern-Simons line associated to $\rho_S$, we need the intermediate object
$$
H(\rho_S):=d^{-1}(c\circ \rho_S)/B^2_S\subset C^2_S/B^2_S.
$$
This is a torsor for
$$
H^2_S:=\prod_{v\in S} H^2(\pi_v, {\zmod n} )\simeq \prod_{v\in S} \nZ
$$
(\cite[Theorem (7.1.8)]{NSW}).
We then use the sum map
$$
\Sigma: \prod_{v\in S} \nZ\to \nZ
$$
to push this out to a $\nZ$-torsor. That is, define
\begin{equation} \label{localtorsor}
L(\rho_S):=\Sigma_*[ H(\rho_S)  ].
\end{equation}
The natural map $H(\rho_S)\to L(\rho_S)$ will also be denoted by the sum symbol $\Sigma$.
\ms

In fact, $L$ extends to a functor from $\mathcal{M}_S^{loc}(A)$ to the category of $\nZ$-torsors.
To carry out this extension,  we just need to extend $H$ to a functor to $H^2_S$-torsors.
 According to Appendices \ref{sec:Appendix A} and \ref{sec:appendix B}, for $a=(a_v)_{v\in S}\in A^S$ and each $v$, there is an element
$h_{a_v}\in C^2(A, {\zmod n})/B^2(A, {\zmod n})$ such that
$$
c\circ \Ad_{a_v}=c+ dh_{a_v}.
$$
Also, 
$$
h_{a_vb_v}=h_{a_v}\circ \Ad_{b_v}+ h_{b_v}.
$$
Hence, given $a: \rho_S\to \rho_S',$  so that $\rho_S'=\Ad_a \circ \rho_S$, we define
$$
H(a): H(\rho_S)\to H(\rho_S')
$$
to be the map induced by
$$
x\mapsto x'=x+(h_{a_v}\circ \rho_v)_{v\in S}.
$$
Then
$$
dx'=dx+(d(h_{a_v}\circ \rho_v))_{v\in S}=(c\circ \rho_v)_{v\in S}+((dh_{a_v})\circ \rho_v)_{v\in S}=(c\circ \Ad_{a_v}\circ \rho_v)_{v\in S}.
$$
So 
$$
x'\in d^{-1}(c\circ \rho'_S)/B^2_S,
$$ 
and by the formula above, it is clear that $H$ is a functor. That is, $ab$ will send $x$ to
$$
x+h_{ab}\circ \rho_S,
$$
while if we apply $b$ first, we get 
$$
x+h_b\circ \rho_S\in H(\Ad_b\circ \rho_S),
$$ 
which then goes via $a$ to
$$
x+h_b\circ \rho_S +h_a\circ \Ad_b\circ \rho_S.
$$
Thus, $$H(ab)=H(a)H(b).$$ Defining
$$
L(a)=\Sigma_*\circ H(a)
$$
turns $L$ into a functor from $\mathcal{M}_S^{loc}$ to $\nZ$-torsors. Even though we are not explicitly laying down geometric foundations, it is clear that $L$ defines thereby an $A^S$-equivariant $\nZ$-torsor on $Y_S^{loc}(A)$, or a $\nZ$-torsor on the stack $\msloc(A)$.
\ms

We can compose the functor $L$ with the restriction
$r_S: \mathcal{M}_S(A)\to \mathcal{M}_S^{loc}(A)$ to get an $A$-equivariant functor
$L^{glob}$ from $Y_S(A)$ to $\nZ$-torsors. 

\begin{lem}
Let $\rho \in Y_S(A)$ and $a\in \Aut(\rho).$ Then $L^{glob}(a)=0$.
\end{lem}

\begin{proof}
By assumption, $\Ad_a\rho=\rho$, and hence, $dh_a\circ \rho=0$. That is,
$h_a\circ \rho\in H^2(\pi_S, \zmod n).$  Hence, by the reciprocity law for $H^2(\pi_S, \zmod n)$  (\cite[Theorem (8.1.17)]{NSW}), we get
$$
\Sigma_*(h_a\circ \rho )=0.
$$
\end{proof}
By the argument of  \cite[p. 439]{FQ}, we see that there is a $\nZ$-torsor
$$
L^{\inv}([\rho])
$$ 
of invariant sections for the functor $L^{glob}$ depending only on the orbit $[\rho]$. This is the set of families of elements 
$$
x_{\rho'}\in L^{glob}(\rho')
$$ 
as $\rho'$ runs over $ [\rho]$ with the property  that every morphism $a: \rho_1\to \rho_2$ takes $x_{\rho_1}$ to $x_{\rho_2}$.  Alternatively, $L^{\inv}([\rho])$ is the inverse limit of the $L^{glob}(\rho')$ with respect to the indexing category $[\rho]$. 

\ms

Since 
$$
H^3(\pi_S, \zmod n)=0
$$
(\cite[Proposition (8.3.18)]{NSW}), the cocycle
$c\circ \rho$ is a coboundary 
\begin{equation}\label{global trivialisation}
c\circ \rho=d\b
\end{equation}
for
$\b\in C^2(\pi_S, \zmod n)$. This element defines a class 
\begin{equation} \label{CSA with boundary}
CS_c([\rho]):=\Sigma([i^*_S(\b)])\in L^{\inv}([\rho]).
\end{equation}
A different  choice $\b'$ will be related by 
$$
\b'=\b+z
$$
for a 2-cocycle $z\in Z^2(\pi_S, \zmod n)$, which vanishes when mapped to $L((\rho\circ i_v)_{v\in S})$ because of the reciprocity sequence
\[
\xyh{3}
\xymatrix{
0 \ar[r] &  H^2(\pi_S, \zmod n) \ar[r] & H_S^2 \ar[r]^-{\sum_v \inv_v} & \nZ \ar[r] & 0.
}
\]
 Thus, the class $CS_c([\rho])$ is independent of the choice of $\b$ and defines a global section
$$
CS_c\in \G( \mathcal{M}_S(A), L^{glob}).
$$
Within the context of this paper, a `global section' should just be interpreted as an assignment of $CS_c([\rho])$ as above for each orbit $[\rho]$.

\section{The arithmetic Chern-Simons action: the $p$-adic case} \label{sec:p-adic case}

%

  Now fix a prime $p$ and assume all primes of $F$ dividing $p$ are contained in $S$.   Fix a compatible system $(\zeta_{p^n})_n$ of $p$-power roots of unity, giving us an isomorphism 
$$
\z: \Z_p\simeq \Z_p(1):=\invlim_n \mu_{p^n}.
$$
In this section, we will be somewhat more careful with this isomorphism. Also, it will be necessary to make some assumptions on the representations that are allowed. 
\ms

Let $A$ be a $p$-adic Lie group, e.g., $GL_n(\Z_p)$. Assume $A$ is equipped with an open  homomorphism
$t:A \to \G:=\Z_p^\xx$ and define $A^n$ to be the kernel of the composite map 
$$
A\to \Z_p^\xx\to (\Z/p^n\Z)^\xx=:\G_n.
$$
Let 
$$
\Ainfty=\cap_n A^n=\Ker(t).
$$ 

In this section, we denote by $Y_S(A)$ the continuous homomorphisms $$\rho: \pi_S\to A$$ such that $t\circ \rho$ is a  power $\chi^s$ of the $p$-adic cyclotomic character $\chi$ of $\pi_S$ by a $p$-adic unit $s$. (We note that $s$ itself is allowed to vary.) Of course this condition will be satisfied by any geometric Galois representations or natural $p$-adic families containing one.
\ms
 
    As before, $A$ acts on $Y_S(A)$ by conjugation. But in this section, we will restrict the action to
$\Ainfty$ and   use the notation $\mathcal{M}_S(A)$ for the corresponding action groupoid.

Similarly, we denote by $Y_S^{loc}$ the collections of continuous homomorphisms
$$\rho_S:=(\rho_v: \pi_v \to A)_{v\in S}$$
for which there exists a $p$-adic unit $s$ such that $t\circ \rho_v=(\chi|_{\pi_v})^s$ for all $v$. $\mathcal{M}^{loc}_S(A)$ then denotes the action groupoid defined by the product   $(\Ainfty)^S$ of the conjugation action on  the $\rho_S$.
\ms

We now fix a continuous cohomology class $$c\in  H^3(A, \Z_p[[\G]]),$$where 
$$\Z_p[[\G]]=\invlim_n\Z_p[\G_n].$$ We represent $c$  by a cocycle in $Z^3(A, \Z_p[[\G]])$, which we will also  denote by $c$.
Given $\rho\in Y_S(A)$, we can view $\Z_p[[\G]]$ as a continuous representation of $\pi_S$, where the action is left multiplication via $t\circ \rho$. We denote this representation by
$\Z_p[[\G]]_{\rho}$. The isomorphism
 $\z:\Z_p\simeq \Z_p(1)$, even though it's not $\pi_S$-equivariant, does induce a $\pi_S$-equivariant isomorphism
$$
\z_{\rho}: \Z_p[[\G]]_{\rho}\simeq \L:=\Z_p[[\G]]\otimes \Z_p(1).
$$
Here, $\Z_p[[\G]]$ written without the subscript refers to the action via the cyclotomic character of $\pi_S$ (with $s=1$ in the earlier notation). The isomorphism is defined as follows.
If $t\circ \rho=\chi^s$, then we have the isomorphism
$$\Z_p[[\G]]\simeq \Z_p[[\G]]_{\rho}$$
that sends $\g$ to $\g^s$. On the other hand, we also have
$$\Z_p[[\G]]\simeq \L$$
that sends $\g$ to $\g\otimes \g\z(1).$ Thus, $\z_{\rho}$ can be taken as the inverse of the first followed by the second.

  \ms

Combining these considerations,  we get an element
$$
\z_{\rho}\circ \rho^*c=\zeta_{\rho}\circ c\circ \rho \in Z^3(\pi_S, \L). 
$$
Similarly, if $\rho_S:=(\rho_v)_{v\in S}\in Y^{loc}_S$, we can regard $\Z_p[[\G]]_{\rho_v}$ as a representation of $\pi_v$ for each $v$, and  we get
$\pi_v$-equivariant isomorphisms
$$
\z_{\rho_v}:\Z_p[[\G]]_{\rho_v}\simeq \L.
$$
We also use the notation $$\z_{\rho_S}:\prod_{v\in S}\Z_p[[\G]]_{\rho_v}\simeq \prod_{v\in S}\L$$
for the isomorphism given by the product of the $\z_{\rho_v}$.

\ms

It will be convenient to again denote by  $C^i_S(\L)$ the product $\prod_{v\in S} C^i(\pi_v, \L)$ and use the similar notations $Z^i_S(\L)$, $B^i_S(\L)$ and $H^i_S(\L)$.
The element $\z_{\rho_S}\circ \rho_S^*c$ is  an element in $ Z^3_S(\L)$.
We then put
$$
H(\rho_S, \L):=d^{-1}((\z_{\rho_S}\circ \rho_S^*c) )/B^2_S(\L)\subset C^2_S(\L)/B^2_S(\L).
$$
This is a torsor for $$H^2_S(\L)\simeq \prod_{v\in S} H^2(\pi_v, \L).$$
The augmentation map $$a:\L\to \Z_p(1)$$ for each $v$ can be used to push this out to a torsor
$$a_*(H(\rho_S, \L))$$
for the group 
$$\prod_{v\in S} H^2(\pi_v, \Z_p(1))\simeq \prod_{v\in S}\Z_p,$$  which then can be pushed out with the sum map $$\Sigma :\prod_{v\in S}\Z_p\to \Z_p$$ to give us a $\Z_p$-torsor $$L(\rho_S, \Z_p):=\Sigma_*(a_*(H(\rho_S, \L))).$$

As before, we can turn this into a functor
$L(\cdot, \Z_p)$ on $\mathcal{M}^{loc}_S(A)$, taking into account the action of $(\Ainfty)^S$. By composing with the restriction functor
$$
r_S:\mathcal{M}_S(A)\to \mathcal{M}^{loc}_S(A),
$$
we also get a $\Z_p$-torsor $L^{glob}(\cdot, \Z_p)$
on $\mathcal{M}_S(A)$.
\ms

We now  choose an element $\b\in C^2(\pi_S, \L)$ such that
$$
d\b=\z_{\rho}\circ c\circ \rho\in Z^3(\pi_S, \L)=B^3(\pi_S, \L)
$$ 
to define the $p$-adic Chern-Simons action
$$
CS_c([\rho]):=\Sigma_*a_*i_S^*(\b)\in L^{glob}([\rho],\Z_p).
$$
The argument that this action is independent of $\b$ and equivariant is also the same as before, giving us an element 
$$
CS_c\in \G(  \mathcal{M}_S(A) ,L^{glob}(\cdot, \Z_p)).
$$

\section{Towards computation: the decomposition formula}\label{sec: towards computation}
In this section, we indicate how one might go about computing the arithmetic Chern-Simons invariant in the unramified case with finite coefficients. That is, we assume we are in the setting of Section \ref{sec:classical case}. 
We provide a proof of Theorem \ref{main theorem} in a slightly generalized setting.
\ms

Let $X=\Spec(\O_F)$ and $M$ a continuous representation of $\pi=\pi_1(X, \fb)$ regarded as a locally constant sheaf on $X$. Assume $M=\invlim M_i$ with $M_i$ finite representations such that there is a finite set $T$  of primes in $\O_F$ containing all primes dividing the order of any $\abs{M_i}$. Let $U=\Spec(\O_{F,~T})$, $\pi_T=\pi_1(U, \fb)$, and $\pi_v=\Gal({\ov F}_v/F_v)$ for a prime $v$ of $F$. Fix natural homomorphisms
\[
\k_T : \pi_T \to \pi \qa \k_v : \pi_v \to \pi.
\]
We denote by $\rho_T$ (resp. $\rho_v$) the composition of $\k_T$ (resp. $k_v$) with 
$$
\rho \in \Hom_{cont}(\pi, M).
$$
Finally, we write $\fP_v$ for the maximal ideal of $\O_F$ corresponding to the prime $v$ and $r_v$ for the restriction map of cochains or cohomology classes from $\pi_T$ to $\pi_v$.
\ms

Denote by
$C^*_c(\pi_T, M)$ the complex defined as a mapping fiber
$$
C^*_c(\pi_T, M):=\mathrm{Fiber}[C^*(\pi_T, M)\to \prod_{v\in T} C^*(\pi_v, M)].
$$
So
$$
C^n_c(\pi_T, M)=C^n(\pi_T,M)\times\prod_{v\in T} C^{n-1}(\pi_v, M),
$$
and
$$
d(a, (b_v)_{v \in T})=(da, (r_v(a)-db_v)_{v \in T})
$$
for $(a,(b_v)_{v \in T})\in C^n_c(\pi_T, M)$.
As in \cite[p. 20]{FK}, since there are no real places in $F$, there is a quasi-isomorphism
$$
C^*_c(\pi_T, M)\simeq R\G(X, j_!j^*(M)),
$$
where  $j:U\to X$ is the inclusion.
But  there is also an exact sequence
$$
\xyh{1}
\xymatrix{
0 \ar[r] &  j_!j^*(M) \ar[r] & M \ar[r] & i_*i^*(M) \ar[r] & 0,
}
$$
where $i:T\to X$ is the closed immersion complementary to $j$. Thus, we get an exact sequence
\begin{align*}
\xyh{1} 
\xymatrix{
\prod\limits_{v\in T}H^2(k_v, i^*(M)) \ar[r] & H^3(C_c^*(\pi_T, M)) \ar[r] & H^3(X, M) \ar[r] & \prod\limits_{v\in T} H^3(k_v, i^*(M)),
}
\end{align*}
where $k_v:=\Spec(\O_F/\fP_v)$, from which we get an isomorphism
$$
H^3_c(U, M):=H^3(C_c^*(\pi_T, M))\simeq H^3(X, M),
$$
since $k_v$ has cohomological dimension 1.

We interpret this as a statement that the cohomology of $X$
$$
H^3(X, M)
$$
can be identified with  cohomology of a `compactification' of $U$ with respect to the `boundary',  that is, the union of the 
$\Spec(F_v)$ for $v\in T$. 
This means that a class $z\in H^3(X, M)$ is represented by
$(a, (b_v)_{v\in T})$, where $a \in Z^3(\pi_T, M)$ and
$b_v\in C^2(\pi_v, M)$ in such a way that 
$$
db_v=r_v(a).
$$
There is also the exact sequence
$$
\xyv{1}
\xymatrix{
\ar[r] & H^2(\pi_T, M) \ar[r] & \prod\limits_{v\in T} H^2(\pi_v, M) \ar[r] & H^3_c(U, M) \ar[r] & 0,
}
$$
the last zero being $H^3(U, M):=H^3(\pi_T, M)=0$.
We can use this to compute the invariant of $z$ when $M=\mu_n$. (Note that $F$ contains $\mu_n$ and hence it is in fact isomorphic to the constant sheaf $\zmod n$.)
We have to lift $z$ to a collection of classes $x_v\in H^2(\pi_v, \mu_n)$ and then take the sum
$$
\inv(z)=\sum_v\inv_v(x_v).
$$
This is independent of the choice of the $x_v$ by the reciprocity law (cf. \cite[p. 541]{mazur}). The lifting process may be described as follows. The map 
$$
\prod_{v\in T} H^2(\pi_v, \mu_n)\rra H^3_c(U, \mu_n)
$$
just takes a tuple of 2-cocycles $(x_v)_{v\in T}$ to $(0, (x_v)_{v\in T})$. But by the vanishing of $H^3(U, \mu_n)$, given
$z=(a, (b_{-,v})_{v \in T})$, we can find a global cochain $b_+\in C^2(\pi_T, \mu_n)$ such that $db_+=a$. We then put
$$
x_v:= b_{-,v}-r_v(b_+).
$$
Note that $(0, (x_v)_{v\in T})$ is cohomologous to $z=(a, (b_{-,v})_{v \in T})$.
\ms

As before, we start with a class $c \in H^3(A, \mu_n) \simeq H^3(A, \zmod n)$.
Then, we get a class
\[
z= j^3 \circ \rho^*(c) \in H^3(X, \mu_n),
\]
where $j^i : H^i (\pi, \mu_n) \to H^i (X, \mu_n)$ is the natural map from group cohomology to \'etale cohomology (cf. \cite[Theorem 5.3 of Chap. I]{Mi}).
Let $w$ be a cocycle representing $\rho^*(c) \in H^3(\pi, \mu_n)$.
Let $I_v \subset \pi_v$ be the inertia subgroup. We now can trivialise $\k_v^*(w)$ by first doing it over $\pi_v/I_v$ to which it factors. That is,
the $b_{-,v}$ as above can be chosen as cochains factoring through $\pi_v/I_v$.
This is possible because $H^3(\pi_v/I_v, \mu_n)=0$. The class $(\k_T^*(w), (b_{-,v})_{v \in T})$ chosen in this way is independent of the choice of the $b_{-,v}$. This is because $H^2(\pi_v/I_v, \mu_n)$ is also zero.
The point is that the representation of $z$ as $(\k_T^*(w), (b_{-,v})_{v\in T})$ with unramified $b_{-,v}$ is essentially canonical. 
More precisely, given $\k_v^*(w)|_{(\pi_v/I_v)}\in Z^3(\pi_v/I_v, \mu_n)$, there is a canonical 
$$
b_{-,v}\in C^2(\pi_v/I_v, \mu_n)/B^2(\pi_v/I_v, \mu_n)
$$
such that $db_{-,v}=\k_v^*(w) |_{(\pi_v/I_v)}$. This can then be lifted to a canonical class
in 
$$
C^2(\pi_v, \mu_n)/B^2(\pi_v, \mu_n).
$$
 Now we trivialise $\k_T^*(w)$ globally as above,
that is, by the choice of $b_+\in C^2(\pi_T, \mu_n)$ such that
$db_+=\k_T^*(w)$. Then
$(b_{-,v}-b_{+,v})_{v\in T}$ will be cocycles, where $b_{+, v}:=r_v(b_+)$, 
and we compute
\[ 
\inv(z)=\sum_{v \in T} \inv_v (b_{-,v}-b_{+, v}).
\]
Thus, for a given homomorphism $\rho : \pi \to A$, it suffices to find various trivialisations of $\rho^*(c)$ after restriction to $\pi_T$ and to $\pi_v$ for $v \in T$.
\begin{itemize}
\item
We are free to choose a finite set $T$ of primes in a convenient way as long as $T$ contains all primes dividing $n$. And then, for any $v\in T$, solve
\[
db_{-,v}=\rho_v^*(c) \in Z^3 (\pi_v, \mu_n).
\] 
In fact, $b_{-,v}$ comes from an element in $C^2(\pi_v/{I_v}, \mu_n)$ by inflation, so $b_{-,v}$ is unramified.
\item
For chosen $T$, solve
\[
db_+=\rho_T^*(c) \in Z^3(\pi_T, \mu_n),
\]
and we set $b_{+, v} = r_v (b_+) \in C^2(\pi_v, \mu_n)$. 
\end{itemize}
Then, we have the decomposition formula
\[\label{eqn:decomposition formula}
CS_{c} ([\rho])=\sum\limits_{v \in T} \inv_v ([b_{-,v}-b_{+,v}]) \tag{$\dagger$}.
\]
In the case $M=\mu_n$ and $S=T$, a finite set of primes in $\cO_F$ containing all primes dividing $n$, a simple inspection implies that 
\[
\sum\limits_{v \in T} \inv_v ([b_{-,v}-b_{+,v}])= \sum\limits_{v \in S}(\beta_v)- CS_c([\rho \circ \k_S]).
\]
Thus, the formula ($\dagger$) provides a proof of Theorem \ref{main theorem}.
In general, $b_{-, v}$ and $b_{+, v}$ are not cocycles but their difference is.
This corresponds to the fact that $\sum\limits_{v \in S}(\beta_v)$  and $CS_c([\rho \circ \k_S])$ are not an element of $\nZ$ but their difference is.
\ms

A few remarks about this method:
\ms

1. Underlying this is the fact that the compact support cohomology $H^3_c(U, \mu_n)$ can be computed relative to the somewhat fictitious boundary of $U$ or as relative cohomology $H^3(X, T; \mu_n).$ Choosing the unramified local trivialisations corresponds to this latter representation.
\ms

2. To summarise the main idea again, starting from a cocycle $ z \in Z^3(\pi, \mu_n)$ we have canonical unramified trivialisations at each $v$ and a non-canonical  global ramified trivialisation. \bq {\em The invariant of $z$ measures the discrepancy between the unramified local trivialisations and a ramified global trivialisation. }\eq
The fact that the non-canonicality of the global trivialisation is unimportant follows from the reciprocity law (cf. \cite[p. 541]{mazur}). 
\ms

3. The description above that computes the invariant by comparing the local unramified trivialisation with the global ramified one is a precise analogue of the so-called `gluing formula' for Chern-Simons invariants when applied to $\rho^*(c)$ for a representation $\rho: \pi\to \zmod n$ and a 3-cocycle $c$ on $\zmod n$. 
\ms

\section{Examples}\label{sec: examples}
In this section, we provide several explicit examples of computation of $CS_{c}([\rho])$. We still assume that we are in the setting of Section \ref{sec:classical case}.

\subsection{General strategy}\label{sec:general strategy}
To compute the arithmetic Chern-Simons invariants, we essentially use the decomposition formula (\ref{eqn:decomposition formula}) in Section \ref{sec: towards computation}.
The most difficult part in the above method is finding an element $b_+ \in C^2(\pi_T, \mu_n)$ that gives a global trivialisation. 
\ms

To simplify our problem, we assume that
a cocycle $c \in Z^3(A, \mu_n)$ is defined by the cup product:
\[
c = \a \cup \e,
\] 
where $\a \in Z^1(A, \mu_n)=\Hom(A, \mu_n)$ and $\e \in Z^2(A, \zmod n)$ is a cocycle representing
an extension 
\[
\xymatrix{
E : 0 \ar[r] & \zmod n \ar[r] & \Gamma \ar[r]_-\p & A \ar[r] & 1.
}
\]
We note that if we take a section $\sigma$ of $\p$ that sends $e_A$ to $e_{\Gamma}$, then
\[
\e (x, y)=\sigma(x)\cdot\sigma(y)\cdot\sigma(xy)^{-1} \in \Ker \p = \zmod n
\]
(cf. \cite[p. 183]{weibel}).
As discussed in Section \ref{sec:classical case}, this assumption is vacuous if $A=\zmod n$.
\ms

To find $b_{-,v}$ and $b_{+,v}$ in the decomposition formula (\ref{eqn:decomposition formula}), we first trivialise $\e$ in $\pi_v$ and $\pi_T$, respectively. 
Namely, let 
\[
d \g_{-, v}=\rho_v^*(\e) \qa d\g_+ =  \rho_T^* (\e).
\]
Here, the precise choice of $\g_{-,v}$ will be unimportant, except it should be unramified and normalised so that $\g_{-,v}(e_A)=0$. Hence, we will be inexplicit below about this choice. 
Again, let $\g_{+,v}=r_v(\g_+)$. 
Then, we have
\[
d(\rho_v^*(\a) \cup \g_{-,v}) = -\rho_v^*(\a) \cup d\g_{-,v} = - \rho_v^* (\a \cup \e)=-\rho_v^*(c)
\]
and
\[
d(\rho_T^*(\a) \cup \g_{+}) = -\rho_T^*(\a) \cup d\g_{+} = - \rho_T^* (\a \cup \e)=-\rho_T^*(c).
\]
Therefore, we can find
\[
b_{-, v}=-\rho_v^*(\a) \cup \g_{-, v} \qa b_{+,v}=r_v(b_+)=r_v(-\rho_T^*(\a) \cup \g_{+})= -\rho_v^*(\a) \cup \g_{+, v}.
\]
\ms

In summary, we get the following formula.
\begin{thm}\label{thm:decomposition formula}
For $\rho$ and $c$ as above, we have
\begin{equation}
CS_c([\rho]):=CS_{[c]}([\rho])=\sum\limits_{v \in T} \inv_v (\rho_v^* (\a) \cup \psi_v),
\end{equation}
where $\psi_v = \g_{+, v} - \g_{-, v} \in Z^1(\pi_v, \zmod n)=H^1(\pi_v, \zmod n)=\Hom (\pi_v, \zmod n)$. 
\end{thm}
So, to evaluate the arithmetic Chern-Simons action, we need to study 
\begin{itemize}
\item
a trivialisation of certain pullback of a 2-cocycle $\e$, and
\item
the local invariant of a cup product of two characters on $\pi_v$.
\end{itemize}
\ms

In the following two subsections, we will see how this idea can be realised.

\subsection{Trivialisation of a pullback of $\e$} \label{sec:trivialisation}
As before, let $\e \in Z^2 (A, \zmod n)$ denote a 2-cocycle representing an extension
\[
\xymatrix{
E : ~0 \ar[r] & \zmod n \ar[r] & \Gamma \ar[r]_-{\p} & A \ar@/_1pc/@{-->}[l]_-{\sigma} \ar[r] & 1
}
\]
with a section $\sigma$ such that $\sigma(e_A)=e_\Gamma$. 
\ms

Suppose that we have the following commutative diagram of group homomorphisms:
\[ \label{eqn:key diagram}
\vcenter{\xymatrix{
& \Ker (f) ~~\ar@{^(->}[r] \ar@{-->}[dl]_-{\wt{f}|_{\Ker (f)}}& \wt{A} \ar[d]^-{f} \ar[dl]_-{\wt{f}} \\
\zmod n ~~\ar@{^(->}[r] &\Gamma \ar[r]_-{\p} & A. 
}}\tag{$\star$}
\]
Then, we can easily trivialise $f^* (\e) \in Z^2(\wt{A}, \zmod n)$.

\begin{lem}\label{lem:trivialisation}
For any $g \in \wt{A}$, let
\[
\g(g) := \sigma(f(g)) \cdot \wt{f}(g)^{-1}.
\]
Then, $\gamma(g) \in \Ker (\p)=\zmod n$ and $d\g = f^* (\e) \in Z^2(\wt{A},  \zmod n)$. Furthermore, we have $\g(e_{\wt{A}})=0$ and $\g(g\cdot h)=\g(g) + \g(h)$ for any $g, h \in \Ker (f)$.
\end{lem}

\begin{proof}
First, we note that $\g(g) \in \Ker (\p)$ because $\p \circ \sigma$ is the identity and $\p \circ \wt{f}=f$.
By definition and the fact that $\Ker (\p)$ is in the center of $\Gamma$,
\begin{align*}
d\g(x, y)&=\g(y) \cdot \g(xy)^{-1} \cdot \g(x) =\g(y) \cdot \g(x) \cdot \g(xy)^{-1} \\
 &=\{\sigma(f(y)) \cdot \wt{f}(y)^{-1}\} \cdot \{ \sigma (f(x))\cdot\wt{f}(x)^{-1}  \} \cdot \{\sigma(f(xy))\cdot\wt{f}(xy)^{-1} \}^{-1} \\
 &= \{\sigma(f(y)) \cdot \wt{f}(y)^{-1}\} \cdot \sigma (f(x))\cdot\wt{f}(x)^{-1} \cdot \wt{f}(x) \cdot \wt{f}(y) \cdot \sigma(f(xy))^{-1} \\
 &= \sigma (f(x)) \cdot \{\sigma(f(y)) \cdot \wt{f}(y)^{-1}\} \cdot \wt{f}(y) \cdot \sigma(f(xy))^{-1} \\
 &=\sigma(f(x)) \cdot \sigma(f(y)) \cdot \sigma(f(x \cdot y))^{-1}\\
 &=f^*(\e) (x, y).
\end{align*}
Therefore the first claim follows. Also, $\g(e_{\wt{A}})=0$ because $\sigma(f(e_{\wt{A}}))=\sigma(e_{A})=e_\Gamma$ and $\wt{f}(e_{\wt{A}})=e_\Gamma$. Finally, for any $g\in \Ker (f)$, $\g(g)=-\wt{f}(g)$, so it is a homomorphism because $\wt{f}$ is a homomorphism and 
the image of $\wt{f}|_{\Ker (f)}$, which is contained in $\zmod n$, is abelian. 
\end{proof}
\ms

\begin{rem}\label{rem:key diagram choice}
In Diagram (\ref{eqn:key diagram}), we can take $\wt{A}=\Gamma$, $f=\p$ and $\wt{f}$ is the identity. For the rest of this section, we always fix such a choice.
\end{rem}

\subsection{Local invariant computation} \label{sec:local invariant computation}
In this subsection, we investigate several conditions to ensure
\[
\inv_v (\phi \cup \psi) \neq 0 \in \nZ,
\]
where $\phi \in H^1(\pi_v, \mu_n)=\Hom(\pi_v, \mu_n)$ and $\psi \in Z^1(\pi_v, \zmod n)=\Hom (\pi_v, \zmod n)$.

\begin{lem} \label{lem:trivial}
Suppose that $\phi$ is unramified, i.e., $\phi$ factors through $\pi_v/{I_v}$. Then, 
\[
\inv_v (\phi \cup \psi ) = 0
\]
if one of the following holds.
\begin{enumerate}
\item
$\phi = 1$, the trivial character.
\item
$\psi$ is unramified.
\end{enumerate}
\end{lem}
\begin{proof}
If $\phi = 1$, then $\phi \cup \psi = 0 \in H^2(\pi_v, \mu_n)$. Thus, $\inv_v (\phi \cup \psi ) =0$.
Also, if $\psi$ is unramified, then $\phi \cup \psi$ arises from $H^2(\pi_v/{I_v}, \mu_n)$ by inflation, which is 0.
Therefore, $\phi \cup \psi = 0 \in H^2(\pi_v, \mu_n)$ and the result follows.
\end{proof}

If $v$ does not divide $n$, then we can prove more.
\begin{lem} \label{lem:non-trivial}
Assume that $v$ does not divide $n$. And assume that $\phi$ is an unramified generator of $\Hom (\pi_v, \mu_n)$, i.e., a generator of $\Hom(\pi_v/I_v, \mu_n)$. Then,
\[
\inv_v (\phi \cup \psi) \neq 0 \Longleftrightarrow \psi \text{ is ramified}.
\]
\end{lem}
\begin{proof}
Using a fixed primitive $n$-th root $\z$ of unity, we fix an isomorphism
\[
\xyv{0.2}
\xymatrix{
\eta : \zmod n \ar[r] & \mu_n \\
\hs{7} a \hs{3} \ar@{|->}[r] & \z^a 
}
\]
and using ${\eta}$, we get natural isomorphisms
\[
\xymatrix{
\Hom(\pi_v, \nZ) & \ar[l]_-{\frac{1}{n}\cdot (-)}\Hom (\pi_v, \zmod n ) \ar@/^1pc/[r]^-{{\eta} \circ (-)} & \ar@/^1pc/[l]^-{{\eta}^{-1} \circ (-)}\Hom (\pi_v, \mu_n).
}
\]
In this proof, we will regard $\phi$ as an element of $\Hom (\pi_v, \nZ)$ and $\psi$ as one of $\Hom (\pi_v, \mu_n)$
using the above isomorphisms. 

\vspace{2mm}
If $\psi$ is unramified, $\inv_v (\phi \cup \psi )=0$ by the above lemma. Since $\mu_n \subset F_v$,
by the Kummer theory we can find an element $a \in F_v^*$ such that $\delta (a) = \psi$,
where $\delta : F_v^*/{(F_v^*)^n} \simeq H^1(\pi_v,~\mu_n) = \Hom (\pi_v, \mu_n)$. Let 
\[
\ord_v : F_v^* \rra \Z
\]
be the normalized valuation on $F_v^*$ that sends a uniformiser 
$\vp$ of $\cO_{F_v}$ to $1$. Then, 
\[
\psi  \text{ is ramified } \Longleftrightarrow ~~\ord_v(a) \not\equiv 0 \pmod n.\footnote{This is where our assumption that $v\nmid n$ is used.}
\]
Since $\phi$ is an unramified generator, $\phi(\Frob)=\frac{t}{n}$ for some $t \in (\zmod n)^\times$, where $\Frob$ is a lift of the Frobenius in $\pi_v/I_v$ to $\pi_v$. Then,
\[
\inv_v(\phi \cup \psi)=\inv_v( \phi \cup \d (a)) =\phi(\Frob^{\ord_v(a)})=\frac{t \cdot \ord_v(a)}{n}.
\]
Combining the above two results, we obtain 
\[
\psi  \text{ is ramified } \Longleftrightarrow ~~ \inv_v(\phi \cup \psi) \neq 0
\]
as desired.
\end{proof}

\begin{rem}
When $n=2$, the above lemmas are enough for the computation of local invariants. 
\end{rem}

\subsection{Construction of examples}
From now on, we assume that $n=2$. 
\ms

As a corollary of Section \ref{sec:trivialisation}, if we have the following commutative diagrams
\[ \label{eqn:key diagram2}
\vcenter{
\xyv{0.5}
\xymatrix{
\pi_T \ar@{->>}[r]^-{\wt{\rho_+}}\ar@{->>}[dd]_-{\k_T} \ar[rdd]^-{\rho_+}  & \Gamma \ar@{->>}[dd]^-{\p} & & \pi_v \ar[dd]_-{\k_v} \ar@{->>}[r]^-{\wt{\rho_v}} \ar[rdd]^-{\rho_v} & \Gamma \ar@{->>}[dd]^-\p\\
&&\text{and}&& \\
\pi \ar@{->>}[r]_-\rho & A & & \pi \ar@{->>}[r]_-\rho & A, 
}}\tag{$\star\star$}
\]
then we get 
\[
\g_+=(\wt{\rho_+})^*(\g) \qa \g_{-, v}=(\wt{\rho_v})^*(\g).
\]
\ms

Thus we can explicitly compute $CS_c([\rho])$ using the previous strategy when we are in the following situation:

\begin{assu} \label{assumptions}$~$
\begin{enumerate}
\item
$F$ is a totally imaginary field.
\item
$c=\a \cup \e$ with $\a : A \to \mu_2$ surjective, and $\e$ representing an extension
\[
\xymatrix{
E : ~0 \ar[r] & \zmod 2 \ar[r] & \Gamma \ar[r] & A  \ar[r] & 1.
}
\]
\item
There are Galois extensions of $F$:
\[
F \subset F^{\a} \subset F^\ur \subset F^+
\]
such that 
\begin{itemize}
\item
$\Gal(F^\ur/F)$ is isomorphic to $A$ and $F^\ur/F$ is unramified everywhere.
\item
$\Gal(F^+/F)$ is isomorphic to $\Gamma$ and $F^+/F$ is unramified at the primes above $2$.
\item
$F^{\a}$ is the fixed field of the kernel of the composition 
\[
\Gal(F^\ur/F) \overset{\sim}{\rra} A \overset{\a}{\rra} \mu_2
\]
and hence we get a commutative diagram
\[
\xymatrix{
 & A \ar@/^1pc/[rrd]^-{\a}  \ar@/^0.5pc/[rd] & & \\
\pi \ar@{->>}[ur]^-{\rho} \ar@{->>}[r] & \Gal(F^\ur/F) \ar@{->>}[r] \ar[u]_-\simeq & \Gal(F^{\a}/F) \ar[r]^-{\simeq} & \mu_2. 
}
\]
\end{itemize}
\end{enumerate}
\end{assu}
\ms

Suppose we are in the above assumption.
Let $S$ be the set of primes of $\cO_F$ ramified in $F^+$, and $S_2$ the set of primes of $\cO_F$ dividing $2$.
Then by our assumption, $S\cap S_2 = \emptyset$.
Let $T=S \cup S_2$. Then, we can find a global trivialisation $\g_+$ of $\rho_T^*(\e)$ from the following commutative diagram
\[
\xymatrix{
& \zmod 2 \simeq \Ker ({\phi})=\Gal(F^+/{F^\ur}) ~~\ar[r] \ar@{-->}[dl]_-{\wt{{\phi}}|_{\Ker ({\phi})}=\text{Id}}& \Gal(F^+/F) \ar@{->>}[d]^-{\phi} \ar[dl]_-{\wt {\phi} = \text{Id}} \\
\zmod 2 \ar[r] &\Gamma \simeq \Gal(F^+/F) \ar@{->>}[r] & A\simeq \Gal(F^\ur/F). 
}
\]
For each $v \in T$, let $D(v)$ be the decomposition group of $\Gal(F^+/F)$ at $v$. In other words,
\[
D(v)=\{ g \in \Gal(F^+/F) : gv=v \}  \simeq \Gal(F^+_\nu/{F_v}),
\]
where $\nu$ is a prime of $F^+$ lying above $v$. And let $I(v)$ be the inertia subgroup of $D(v)$. Then, $I(v)=0$ if and only if $v$ divides $2$. Thus,
\[
\g_{+, v} \text{ is unramified} \Longleftrightarrow v \in S_2.
\] 
Since $\psi_v:=\g_{+, v}-\g_{-, v}$ and we always take $\g_{-, v}$ unramified,
\[
\psi_v \text{ is unramified} \Longleftrightarrow v \in S_2.
\] 
Furthermore, 
\[
\rho_v^*(\a) \text{ is trivial} \Longleftrightarrow f(D(v))=0,
\]
where $f$ is the natural projection from $\Gal(F^+/F)$ to $\Gal(F^{\a}/F)$. And $f(D(v))=0$ exactly occurs when $v$ splits in $F^{\a}$. Note that $\rho_v^*(\a)$ is always an unramified generator of $\Hom(\pi_v, \mu_2)$ if it is not trivial.

\ms

Now we are ready to compute the arithmetic Chern-Simons invariants.
\begin{thm}\label{thm:CS formula}
Suppose we are in Assumption \ref{assumptions}. Then,
\[
CS_c([\rho]) = \sum\limits_{v \in T} \inv_v (\rho_v^*(\a) \cup \psi_v)=\frac{r}{2} ~\mod \Z,
\]
where $\psi_v=\g_{+,v}-\g_{-, v}$ and $r$ is the number of primes in $S$ which are inert in $F^{\a}$.
\end{thm}
\begin{proof}
The first equality follows from Theorem \ref{thm:decomposition formula}. Thus, it suffices to compute $\inv_v(\rho_v^*(\a) \cup \psi_v)$ for $v \in T$. By Lemma \ref{lem:trivial}, 
$\inv_v(\rho_v^*(\a) \cup \psi_v)=0$ if either $\rho_v^*(\a)$ is trivial or $\psi_v$ is unramified.
By the above discussion, $\rho_v^*(\a)$ is trivial if and only if $f(D(v))=0$, i.e., $v$ splits in $F^{\a}$; and $\psi_v$ is unramified if and only if $v\in S_2$. Furthermore, if $\rho_v^*(\a)$ is not trivial and 
$\psi_v$ is ramified, then by Lemma \ref{lem:non-trivial}, $\inv_v(\rho_v^*(\a) \cup \psi_v)=\frac{1}{2}$. Thus the result follows.
\end{proof}
\ms

Therefore to provide an example of calculation of the arithmetic Chern-Simons invariants, it suffices to construct a tower of fields satisfying Assumption \ref{assumptions}, which is essentially the embedding problem in the inverse Galois theory. Instead, we will consider the similar problems over $\Q$, which are much easier to solve (or find from the table). Then, we will construct a tower satisfying Assumption \ref{assumptions} from a tower of fields over $\Q$.

\ms

\begin{assu}\label{assu:fields}
Suppose we have a number field $L$ with its subfield $K$ such that
\begin{enumerate}
\item
$\Gal(L/\Q) \simeq \Gamma$.
\item
$d_L$, the (absolute) discriminant of $L$, is an odd integer\footnote{We may consider when $d_L$ is even. Then later, it is not clear that $FL/FK$ is unramified at the primes above $2$. Some choices of $t$ (for $F$) can make it ramified. Then, it is hard to determine the value of local invariants unless $2$ splits in $F^{\a}/F$.}.
\item
$\Gal(K/\Q)\simeq A$.
\item
$\Q(\sqrt{D})$ is a quadratic subfield of $K$, where $D$ is a divisor of $d_K$.\footnote{Here, we always take that $d_K$ is odd because we cannot use Abhyankar's lemma when $p=2$, and hence we may not remove ramification in the extension $FK/F$ at the primes above $2$. In some nice situation, we may directly prove that $F(\sqrt{D})/F$ is unramified at the primes above $2$ even though $D$ is even. If so, our assumption on $d_K$ can be removed.}
\item
$K/{\Q(\sqrt{D})}$ is unramified at any finite primes.
\end{enumerate}
\end{assu}

Then, we have the following.
\begin{prop}
Let $F=\Q(\sqrt{-\abs D \cdot t})$ be an imaginary quadratic field, where $t$ is a positive squarefree integer prime to $D$ so that $F\cap L=\Q$. 
Then, there is a tower of fields $F \subset F^\ur \subset F^+$ satisfies Assumption \ref{assumptions}.
In fact, we can take 
\[
F^\ur=KF \qa F^+=LF.
\]
\end{prop}
\begin{proof}
First, it is clear that $F$ is totally imaginary. 
Next, since $F\cap L=\Q$
\[
\Gal(LF/F)\simeq \Gal(L/\Q) \simeq \Gamma \qa \Gal(KF/F)\simeq \Gal(K/\Q)\simeq A.
\]
Since the discriminant of $L$ is odd, $L/K$ is unramified at the primes above $2$, and so is
$LF/{KF}$. Finally, it suffices to show that $KF/F$ is unramified everywhere. 
Since $K/{\Q(\sqrt{D})}$ is unramified everywhere, $K/\Q$ is only ramified at the primes dividing $D$.
(Note that the discriminant of $K$ is odd, hence it is unramified at $2$.) Moreover, the ramification degree of any prime divisor $p$ of $D$ is $2$, and the same is true for $F/\Q$.
 Since $p$ is odd, $KF/F$ is unramified at the primes above $p$ by Abhyankar's lemma \cite[Theorem 1]{cornell}, which implies our claim.
\end{proof}

\begin{rem}\label{rem:Abhyankar's lemma}
Since the ramification indices of any prime divisor $p$ of $D$ are $2$ in both $F/\Q$ and $K/\Q$, we can use Abhyankar's lemma in both directions. (Note that our assumption implies that $D$ is odd.) In other words, $KF/K$ is always unramified at the primes dividing $D$. 
\end{rem}

The remaining part to check Assumption \ref{assumptions} is the choice of $F^{\a}$.
Let 
\[
B:=\{ F_1, \dots, F_m\}
\] 
be the set of quadratic subfields of $F^\ur$. 
Then, there is one-to-one correspondence between the set of surjective homomorphisms $\Gal(F^\ur/F) \to \mu_2$ and $B$. 
Therefore $m=\# \Hom(A, \mu_2)-1$ and we can define $\a_i : A \to \mu_2$ so that $F^{\a_i}=F_i$ 
due to the (chosen) isomorphism $\Gal(F^\ur/F) \simeq A$.

\ms
Now, suppose $F^{\a}=F(\sqrt{M}) \subset F^\ur$ for some divisor $M$ of $D$. Let $\Q_1=\Q(\sqrt{M})$ and $\Q_2=\Q(\sqrt{N})$, where $N=(-\abs D \cdot t)/M$. Then, we have the following commutative diagram:
\[
\xyh{2.5}
\xymatrix{
& F^{\a}=F(\sqrt{M})=F(\sqrt{N}) \ar@{-}[d]^-{\text{unramified}} \ar@{-}[dr]\ar@{-}[dl]& \\
\Q_1=\Q(\sqrt{M}) &F=\Q(\sqrt{MN})&\Q_2=\Q(\sqrt{N})\\
& \Q \ar@{-}[ur] \ar@{-}[ul] \ar@{-}[u]&
}
\]
For a prime $p$, let $\wp$ denote a prime of $\cO_F$ lying above $p$.
We want to understand the splitting behaviour of $\wp$ in ${F^{\a}}$.
\begin{lem}\label{lem:splitting behaviour}
Let $p$ be an odd prime.
\begin{enumerate}
\item
Assume that $p$ divides $Dt$. Then 
\[
\wp \text{ is inert in } {F^{\a}} \Longleftrightarrow p \text{ is inert either in } \Q_1 \text{ or in } \Q_2.
\]
\item
If $p$ is inert in $F$, then $\wp$ always splits in $F^{\a}$.

\item
Assume that $p$ splits in $F$. Then
\[
\wp \text{ splits in } {F^{\a}} \Longleftrightarrow p \text{ splits in } \Q_1.
\]

\end{enumerate}
\end{lem}
\begin{proof}$~$
\begin{enumerate}
\item
In this case, $p$ is ramified in $F$, and $p$ is ramified either in $\Q_1$ or in $\Q_2$.
Without loss of generality, let $p$ is ramified in $\Q_2$. Then, $\wp$ is inert in $F^{\a}$ if and only if $p$ is inert in $\Q_1$ from the above commutative diagram.

\item
Let $\leg {a}{b}$ denote the Legendre symbol.
If $p$ is inert in $F$, then $\leg{MN}{p}=-1$. Therefore either $\leg Mp=1$ or $\leg Np=1$. Without loss of generality, let $\leg M p=1$ and $\leg Np=-1$.
Then, $p$ splits in $\Q_1$ and hence there are at least two primes in $F^{\a}$ above $p$. 
Since $\wp$ is the unique prime of $F$ above $p$, $\wp$ splits in $F^\a$.

\item
Since $\leg {MN}p=1$, either $\leg Mp=\leg Np=1$ or 
$\leg Mp=\leg Np=-1$.
If $\leg Mp=-1$, then there is only one prime in $\Q_1$ above $p$. Thus, there are at most two primes in $F^{\a}$ above $p$. Since $p$ already splits in $F$, $\wp$ is inert in $F^{\a}$. On the other hand, if $\leg Mp=1$, then $p$ splits completely in $F^{\a}$ because $p$ splits completely both in $\Q_1$ and $F$. Thus, $\wp$ splits in $F^{\a}$.
\end{enumerate}
\end{proof}

Let $D_L=d_{L}/{d_K^2}$ be the norm (to $\Q$) of the relative discriminant of $L/K$. Then, $L/K$ is precisely ramified at the primes dividing $D_L$, and hence 
\[
S \subset \{ \fp \in \Spec (\cO_F) : \fp \mid D_L \}.
\]
(Note that $S$ is the set of primes in $\cO_F$ that ramify in $F^+$.) Let $s$ be the number of prime divisors of 
$(D_L, D)$, which are inert either in $\Q_1$ or in $\Q_2$.
Then, we have the following.

\begin{thm}\label{thm:final formula}
Assume that we have $\rho$ and $c$ as above. Then,
\[
CS_c([\rho]) \equiv \frac{s}{2} \pmod \Z.
\]
\end{thm}
\begin{proof}
First, we show that 
\[
S=\{ \fp \in \Spec (\cO_F) : \fp \mid D_L ~\text{ but } ~\fp \nmid t \}.
\] 
For a prime divisor $p$ of $D_L$ which does not divide $t$, we show that $KF/K$ is unramified at any primes above $p$, which implies that $LF/KF$ is ramified at the primes above $p$.
If $p$ does not divide $D$, then this is done because $p$ is unramified in $F$. On the other hand, if $p$ divides $D$, $KF/K$ is unramified at the primes above $p$ by Remark \ref{rem:Abhyankar's lemma}. Now, assume that $p$ divides $(D_L, t)$, and let $\wp$ be a prime of $\cO_K$ lying above $p$. Then, $\wp$ is ramified both in $L/K$ and in $KF/K$. 
(Note that since $(t, D)=1$, $K/\Q$ is unramified at $p$ but $F/\Q$ is ramified at $p$.) 
Therefore by the same argument as in Remark \ref{rem:Abhyankar's lemma}, $LF/KF$ is unramified at the primes above $p$, which proves the above claim.

Next, by Theorem \ref{thm:CS formula} it suffices to compute the number of primes in $S$ which are inert in $F^{\a}$.  
Let $\wp \in S$ be a prime above an odd prime $p$. 
Assume that $p$ does not divide $D$. (Then $p$ is unramified in $F$.) If $p$ is inert in $F$, then $\wp$ always splits in $F^{\a}$ by Lemma \ref{lem:splitting behaviour}. If $p$ splits in $F$ and $p\cO_F = \wp \cdot \wp'$, then $\wp$ is inert (in $F^{\a}$) if and only if $\wp'$ is inert.
Therefore to compute the invariant, the contribution from such split primes can be ignored.
So, we may assume that $p$ divides $D$. Then, there is exactly one (ramified) prime $\wp$ in $\cO_F$ above $p$, and our claim follows from Lemma \ref{lem:splitting behaviour}.
\end{proof}

We remark that the computation of $s$ is completely easy 
because $\Q_1/\Q$ and $\Q_2/\Q$ are just quadratic fields. And this also illustrates that we only need information on 
the primes dividing $(D_L, D)$ for the computation.

\subsection{Case 1: cyclic group}
Let $A=\zmod 2$, and $\Gamma=\zmod 4$. Then, we can easily find Galois extensions $L/K/\Q$ in Assumption \ref{assu:fields} by the theory of cyclotomic fields. 
\ms

Let $p$ be a prime congruent to $1$ modulo $4$. Then, 
we can take $L$ as the degree 4 subfield of $\Q(\mu_p)$, and $K=\Q(\sqrt{p})$. Moreover, $d_L=p^3$ and $d_K=p$.
\ms

Let $F=\Q(\sqrt{-p\cdot t})$, where $t$ is a positive squarefree integer prime to $p$. (Then, $F\cap L =\Q$.)
\begin{prop} \label{abelian1}
Let $\rho$ and $c$ be chosen so that $F^{\a}=F^\ur = FK$ and $F^+=FL$. Then, 
\[
CS_c([\rho]) = \frac{1}{2} \Longleftrightarrow \leg {t}{p} =-1.
\]
\end{prop}
\begin{proof}
By Theorem \ref{thm:final formula}, it suffices to check whether $p$ is inert in $\Q(\sqrt{-t})$.
If it is inert, then $CS_c([\rho])=\frac{1}{2}$, and $0$ otherwise. Since $p\equiv 1 \pmod 4$, the result follows.
\end{proof}

\subsection{Case 2: non-cyclic abelian group}
Let $A=V_4:=\zmod 2 \times \zmod 2$, the Klein four group, and $\Gamma=\cQ_8=\cQ$, the quaternion group. To find Galois extensions $L/K/\Q$ in Assumption \ref{assu:fields}, we first study quaternion extensions of $\Q$ in general.

\begin{prop}\label{prop:quaternion}
Let $L/\Q$ be a Galois extension whose Galois group is isomorphic to $\cQ$. Suppose that $d_L$ is odd. 
Let $K$ be a subfield of $L$ with $\Gal(L/K)\simeq \zmod 2$. Then,
\begin{enumerate}
\item
$K=\Q(\sqrt{d_1}, \sqrt{d_2})$ for some positive squarefree $d_1$ and $d_2$.
\item
$d_1 \equiv d_2 \equiv 1 \pmod 4$.
\item
Let $p$ be a prime divisor of $d_1d_2$. Then, $p$ divides $D_L:=d_L/{d_K^2}$.
\end{enumerate}
\end{prop}
\begin{proof}
Since $K$ is a subfield of $L$, $d_K$ is also odd. And since $\cQ$ has a unique subgroup of order $2$, which is normal, $K/\Q$ is Galois and $\Gal(K/\Q)\simeq \zmod 2 \times \zmod 2$. Therefore $K=\Q(\sqrt{d_1}, \sqrt{d_2})$, where $d_1$ and $d_2$ are products of prime discriminants. If $L$ is totally real, then $K$ must be totally real as well. If $L$ is not totally real, then the complex conjugation generates a subgroup of $\Gal(L/\Q)$ of order $2$. Since $\cQ$ has a unique subgroup of order $2$, $K$ must be a fixed field of the complex conjugation, which implies that $K$ is totally real. So, $d_1$ and $d_2$ can be taken as positive squarefree integers. Moreover, since they are products of prime discriminants and odd, $d_1 \equiv d_2 \equiv 1 \pmod 4$.

Finally, let $p$ be a prime divisor of $d_1$, which does not divide $d_2$. Note that ${\Q(\sqrt{d_1})} \subset K \subset L$ and $L/{\Q({\sqrt{d_1}})}$ is a cyclic extension of degree $4$. Since $p$ does not divide $d_2$, $\Q(\sqrt{d_2})/\Q$ is unramified at $p$ and hence $K/{\Q(\sqrt{d_2})}$ is ramified at the primes dividing $p$. 
By \cite[Corollary 3]{lemme}, $L/K$ is ramified at the primes above $p$ and hence $p$ divides $D_L$. 
By the same argument, the claim follows when $p$ is a divisor of $d_2$, which does not divide $d_1$.  
Let $p$ be a prime divisor of $(d_1, d_2)$. Then, since $K=\Q(\sqrt{d_1}, \sqrt{d_2})=\Q(\sqrt{d_1}, \sqrt{d_1d_2})=\Q(\sqrt{d_1}, \sqrt{\frac{d_1d_2}{p^2}})$ and $p$ does not divide $\frac{d_1d_2}{p^2}$, the result follows by the same argument as above.
\end{proof}

Now, let $d_1$ and $d_2$ be two squarefree positive integers such that
\begin{itemize}
\item
$d_1 \equiv d_2 \equiv 1 \pmod 4$.
\item
$(d_1, d_2)=1$.\footnote{This is not a vacuous condition.
In fact, there is a $\cQ$-extension $L$ containing $\Q(\sqrt{21}, \sqrt{33})$ \cite{lmfdb8}.}
\end{itemize} 
\ms

Let $K=\Q(\sqrt{d_1}, \sqrt{d_2})$. Suppose that there is a number field $L$ such that
\begin{itemize}
\item
$L/\Q$ is Galois and $\Gal(L/\Q) \simeq \cQ$.
\item
$L$ contains $K$ and the discriminant $d_L$ of $L$ is odd.
\end{itemize} 
\ms

Let $F=\Q(\sqrt{-d_1d_2 \cdot t})$, where $t$ is a positive squarefree integer prime to $d_1 d_2$.
Then $L \cap F = \Q$ because all quadratic subfields of $L$ are contained in $K$, which is totally real. Since $\Hom(A, \mu_2)$ is of order $4$, there are three quadratic subfield of $FK$ over $F$:
\[
F_1:=F(\sqrt{d_1}), ~F_2:=F(\sqrt{d_2}), \text{and } F_3:=F(\sqrt{d_1d_2})=F(\sqrt{-t}).
\]

\begin{prop}\label{prop:V4 general example}
Let $\rho$ and $c_i=\a_i \cup \e$ be chosen so that $F^{\a_i}=F_i$, $F^\ur = FK$ and $F^+=FL$. Then,
\begin{align*}
CS_{c_1}([\rho]) = \frac{1}{2} &\Longleftrightarrow \prod_{p \mid d_1} \leg {-d_2 \cdot t}{p} \times \prod_{p \mid d_2} \leg {d_1}{p}=-1.\\
CS_{c_2}([\rho]) = \frac{1}{2} &\Longleftrightarrow \prod_{p \mid d_1} \leg {d_2}{p} \times \prod_{p \mid d_2} \leg {-d_1\cdot t}{p}=-1.\\
CS_{c_3}([\rho]) = \frac{1}{2} &\Longleftrightarrow \prod_{p \mid d_1d_2}  \leg {-t}{p}=-1.
\end{align*}
\end{prop}
\begin{proof}
By the above lemma and Theorem \ref{thm:final formula}, it suffices to compute the number of prime divisors of $d_1 d_2$, which are inert in $\Q_1$ or in $\Q_2$. 

First, compute $CS_{c_1}([\rho])$. In this case, $\Q_1=\Q(\sqrt{d_1})$ and $\Q_2=\Q(\sqrt{-d_2\cdot t})$. If $p$ is a divisor of $d_1$, it is inert in $\Q_2$ if and only if
\[
\leg {-d_2 \cdot t}{p}=-1.
\]
Therefore, the number of such prime divisors of $d_1$ is odd if and only if
\[
\prod_{p \mid d_1} \leg {-d_2 \cdot t}{p}=-1.
\]
Similarly, the number of prime divisors of $d_2$, which are inert in $\Q_1$, is odd if and only if
\[
\prod_{p \mid d_2} \leg {d_1}{p}=-1.
\]
Thus, we have
\[
CS_{c_1}([\rho])=\frac{1}{2} \Longleftrightarrow \prod_{p \mid d_1} \leg {-d_2 \cdot t}{p} \times \prod_{p \mid d_2} \leg {d_1}{p}=-1.
\]
The remaining two cases can easily be done by the same method as above.
\end{proof}
\ms

We can find Galois extensions $L/K/\Q$ satisfying the above assumptions from the database.
Here, we take $L/K/\Q$ from the LMFDB \cite{lmfdb} as follows. Let 
\small
\[
g(x)=x^8-x^7+98x^6-105x^5+3191x^4+1665x^3+44072x^2+47933x+328171
\]
\normalsize
be an irreducible polynomial over $\Q$, and $\b$ be a root of $g(x)$. Let 
\[
L=\Q(\b) \qa K=\Q(\sqrt{5}, \sqrt{29}).
\]
So, $d_1=5$ and $d_2=29$. Moreover, $D_L=3^2\cdot 5^2 \cdot 29^2$. 
\ms

Let $F=\Q(\sqrt{-5\cdot 29 \cdot t})$, where $t$ is a positive squarefree integer prime to $5 \cdot 29$. 

\begin{cor}\label{cor:V4 example1}
Let $\rho$ and $c_i=\a_i \cup \e$ be chosen as above. Then,
\begin{align*}
CS_{c_1}([\rho]) = \frac{1}{2} &\Longleftrightarrow \leg t 5 = -1 \Longleftrightarrow t \equiv \pm 2 \pmod 5.\\
CS_{c_2}([\rho]) = \frac{1}{2} &\Longleftrightarrow \leg t {29}=-1.\\
CS_{c_3}([\rho]) = \frac{1}{2} &\Longleftrightarrow \leg t 5 =-  \leg t {29}.
\end{align*}
\end{cor}
\ms

Now, we provide another example. Let $L/K/\Q$ from the the LMFDB \cite{lmfdb1} as follows. Let 
\[
g(x)=x^8-x^7-34x^6+29x^5+361x^4-305x^3-1090x^2+1345x-395
\]
be an irreducible polynomial over $\Q$, and $\b$ be a root of $g(x)$.  Let 
\[
L=\Q(\b) \qa K=\Q(\sqrt{5}, \sqrt{21}).
\]
So, $d_1=5$ and $d_2=21$. Moreover, $D_L=3^2\cdot 5^2 \cdot 7^2$. 
\ms

Let $F=\Q(\sqrt{- 105\cdot t})$, where $t$ is a positive squarefree integer prime to $105$. 

\begin{cor}\label{cor:V4 example2}
Let $\rho$ and $c_i=\a_i \cup \e$ be chosen as above. Then,
\begin{align*}
CS_{c_1}([\rho]) = \frac{1}{2} &\Longleftrightarrow  \leg t 5 =-1 \Longleftrightarrow t \equiv \pm 2 \pmod 5.\\
CS_{c_2}([\rho]) = \frac{1}{2} &\Longleftrightarrow  \leg t 3 = - \leg t 7 \Longleftrightarrow 2, 8, 10, 11, 13, 19 \pmod {21}. \\
CS_{c_3}([\rho]) = \frac{1}{2} &\Longleftrightarrow  \leg t 3 \cdot \leg t 5 \cdot \leg t 7 = -1.
\end{align*}
\end{cor}

\ms

Now, we take $A=V_4$, but $\Gamma=D_4$, the dihedral group of order $8$.
We found $L/K/\Q$ from the LMFDB \cite{lmfdb0} as follows. Let 
\small
\[
g(x)=x^8-3x^7+4x^6-3x^5+3x^4-3x^3+4x^2-3x+1
\]
\normalsize
be an irreducible polynomial over $\Q$, and $\b$ be a root of $g(x)$. Let 
\[
L=\Q(\b) \qa K=\Q(\sqrt{-3}, \sqrt{-7}).
\]
If we take $D=21$, then this choice satisfies Assumption \ref{assu:fields}. Moreover, $d_L=3^6\cdot 7^4$ and $d_K=3^2 \cdot 7^2$. 
\ms

Let $F=\Q(\sqrt{-21 \cdot t})$, where $t$ is a positive squarefree integer prime to $21$. (Then, $F\cap L =\Q$ because all imaginary quadratic subfields of $L$ are $\Q(\sqrt{-3})$ and $\Q(\sqrt{-7})$.)
Since $\Hom(A, \mu_2)$ is of order $4$, there are three quadratic subfield of $FK$ over $F$:
\[
F_1:=F(\sqrt{-3}), ~F_2:=F(\sqrt{-7}), \text{and } F_3:=F(\sqrt{21}).
\]

\begin{prop}\label{prop:V4 example3}
Let $\rho$ and $c_i=\a_i \cup \e$ be chosen so that $F^{\a_i}=F_i$, $F^\ur = FK$ and $F^+=FL$. Then,
\begin{align*}
CS_{c_1}([\rho]) = \frac{1}{2} &\Longleftrightarrow \leg t 3 = -1 \Longleftrightarrow t \equiv 2 \pmod 3.\\
CS_{c_2}([\rho]) = \frac{1}{2} & \h \text{ for all } t.\\
CS_{c_3}([\rho]) = \frac{1}{2} &\Longleftrightarrow \leg t 3 =1 \Longleftrightarrow t \equiv 1 \pmod 3.
\end{align*}
\end{prop}
\begin{proof}
Since $D_L=3^2$, the result follows from Theorem \ref{thm:final formula}.
\end{proof}

\subsection{Case 3: non-abelian group}
Let $A=S_4$, the symmetric group of degree $4$. Then, $H^1(A, \mu_2) \simeq \zmod 2$ and
$H^2(A, \zmod 2) \simeq \zmod 2 \times \zmod 2$. Thus, there is a unique surjective map $\a : A \surj \mu_2$ and three non-trivial central extensions $\Gamma_i$ of $A$ by $\zmod 2$:
\begin{itemize}
\item
$\Gamma_1=2^+S_4 \simeq \GL(2, \F_3)$, the general linear group of degree 2 over $\F_3$.
\item
$\Gamma_2=2^-S_4$, the transitive group `$16T65$' in \cite{paderborn}.
\item
$\Gamma_3=2^{\text{det}} S_4$, corresponding to the cup product of the signature with itself.
\end{itemize}
Let $\e_i$ be a cocycle representing the extension
\[
\xymatrix{
0 \ar[r] & \zmod 2 \ar[r] & \Gamma_i \ar[r] & A=S_4 \ar[r] & 0.
}
\]
\ms

In this subsection, we will consider the first two cases.
There are another descriptions of the groups $\Gamma_1$ and $\Gamma_2$. Let 
\[
\xymatrix{
\cE : 1 \ar[r] & \SL(2, \F_3) \ar[r] & \Gamma \ar[r] & \F_3^{\times} \simeq \zmod 2 \ar[r] & 0.
}
\]
If $\cE$ splits, then $\Gamma \simeq \Gamma_1$, otherwise $\Gamma \simeq \Gamma_2$.

\ms

Let $c=\a \cup \e_1$. (So, $\Gamma=\Gamma_1$.) Suppose $\Q \subset \Q(\sqrt{D}) \subset K \subset L$ is a tower of fields satisfying Assumption \ref{assu:fields}. Let $F=\Q(\sqrt{-\abs D \cdot t})$, where $t$ is a squarefree integer prime to $D$ and greater than $1$. Then, $F \cap L= F \cap \Q(\sqrt{D})=\Q$. (The first equality holds because $\Gamma$ has a unique subgroup of order $24$.)
\begin{prop}
Let $\rho$ and $c$ be chosen so that $F^\a=F(\sqrt{D})$, $F^\ur = FK$ and $F^+=FL$. Then, 
\[
CS_c([\rho]) = 0.
\]
\end{prop}
\begin{proof}
Since the extension 
\[
\xymatrix{
\cE : 1 \ar[r] & \SL(2, \F_3) \ar[r] & \GL(2, \F_3) \ar[r] & \F_3^{\times} \simeq \zmod 2 \ar[r] & 0
}
\]
splits, $\Gal(L/\Q) \simeq \Gal(L/{\Q(\sqrt{D})}) \rtimes \Gal(\Q(\sqrt{D})/\Q)$.

Let $p$ be a prime divisor of $(D_L, D)$. By our assumption, $p$ is odd. Let $I_p$ be an inertia subgroup of $\Gal(L/\Q)\simeq \Gamma=\GL(2, \F_3)$. Since $L/K$ and $\Q(\sqrt{D})/\Q$ are ramified at $p$ but $K/{\Q(\sqrt{D})}$ is not, the ramification index of $p$ in $L/\Q$ is $4$, and $I_p \simeq \zmod 2 \times \zmod 2$.

On the other hand, since $p$ is odd, $L/\Q$ is tamely ramified at $p$ and hence $I_p$ must be cyclic, which is a contradiction. Therefore $(D_L, D)=1$ and hence the result follows by Theorem \ref{thm:final formula}.
\end{proof}
\ms

We can find several examples of such towers from the LMFDB.
Let 
\begin{align*}
g_1(x) &=x^8 -4x^7 +7x^6 +7x^5 -51x^4 +50x^3 +61x^2 -107x-83\\
g_2(x) &=x^4- x -1
\end{align*}
be irreducible polynomials over $\Q$ (\cite{lmfdb2, lmfdb3}), and let $L$ (resp. $K$) be the the splitting field of $g_1(x)$ (resp. $g_2(x)$).
Then, $\Gal(L/\Q) \simeq \GL(2, \F_3)$ and $\Gal(K/\Q)\simeq S_4$. Moreover, $d_L=3^{24}\cdot 283^{24}$ and $d_K=283^{12}$. Thus, $D=-283$ satisfies Assumption \ref{assu:fields}. Note that since the discriminant $D$ of $g_2(x)$ is squarefree, $K/{\Q(\sqrt{D})}$ is unramified everywhere (cf. \cite[p. 1]{kedlaya}).
\ms

Let $F=\Q(\sqrt{-283\cdot t})$, where $t$ is a squarefree integer prime to $283$, and $t>1$. 
\begin{cor}
Let $\rho$ and $c$ be chosen so that $F^{\a}=F(\sqrt{-283})$, $F^\ur = FK$ and $F^+=FL$. Then,
\[
CS_c([\rho]) = 0.
\]
\end{cor}
\vv

Now, we consider another case. Let $c=\a \cup \e_2$. (So, $\Gamma=\Gamma_2$.) Let $L$ be the splitting field of 
\scriptsize
\begin{align*}
&f(x)=x^{16}+5x^{15} -790x^{14}-4654x^{13}+234254x^{12}+1612152x^{11}-33235504x^{10}\\
&-263221982x^9+2331584048x^8+21321377994x^7-74566280958x^6-825209618478x^5\\
&+922238608476x^4+13790070608536x^3-6704968288135x^2-80794234036917x+87192014930816.
\end{align*}
\normalsize
Let $K$ be the splitting field of 
\[
g(x)=x^4 - x^3 - 4x^2 + x +2.
\]
Then, $\Gal(L/\Q) \simeq \Gamma=\Gamma_2$ and $\Gal(K/\Q)\simeq S_4 =A$.\footnote{This example is provided us by Dr. Kwang-Seob Kim.} (See \cite{paderborn, paderborn2}.) 
\begin{lem}
We have the following.
\begin{enumerate}
\item
$K/{\Q(\sqrt{2777})}$ is unramified everywhere.
\item
$\Q(\sqrt{2777})$ is a unique quadratic subfield of $L$.
\item
$\Q(\sqrt{2777}) \subset K\subset L$.
\item
$D_L$ is a multiple of $2777$, i.e., $L/K$ is ramified at the primes above $2777$.
\end{enumerate}
\end{lem}
\begin{proof}
For simplicity, let $E:=\Q(\sqrt{2777})$ and $p=2777$. 
\begin{enumerate}
\item
Since $S_4$ has a unique subgroup of order $12$, $K$ has a unique quadratic subfield $K'$. Since the discriminant of $g(x)$ is $p$, a prime, $K'=E$ and $K/E$ is unramified everywhere (cf. \cite[p. 1]{kedlaya}).

\item
Let $\b_i$ be the roots of $f(x)$. Then, $L=\cup~\Q(\b_i)$.
Since the discriminant of the field $\Q[x]/{(f(x))}$ is $p^{12}$, $\Q(\b_i)$ contains $E$, and so does $L$. On the other hand, since $\Gamma$ has also a unique subgroup of order $24$, $E$ is a unique quadratic subfield of $L$. 

\item\label{3}
Since
\[
f(x)\equiv (x+1372)^4 \cdot (x+1791)^4 \cdot (x+1822)^4 \cdot (x+2653)^4 \pmod {p},
\]
the ramification index of $p$ in $\Q(\b_i)/\Q$ is $4$. Since $L=\cup ~\Q(\b_i)$ and $p$ is odd, the ramification index of $p$ in $L/\Q$ is $4$ by Abhyankar's lemma. Since $L/\Q$ is tamely ramified at $p$, the inertia subgroup $I_p$ of $\Gal(L/\Q)\simeq \Gamma$ is cyclic of order $4$. Since $\Gamma$ has a unique subgroup $C$ of order $2$, $I_p$ contains $C$. Thus, $L/M$ is ramified at the primes above $p$, where $M$ is the fixed field of $C$ in $L$. Since $E/\Q$ is also ramified at $p$, $M/E$ is unramified at the primes above $p$, and hence $M/E$ is unramified everywhere. 
\[
\xyv{0.7} \xyh{0.5}
\xymatrix{
 & L \ar@{-}[d] \ar@{-}@/^1pc/[d]^-{\text{ramified only at the primes above } p}\\
\text{unique $S_4$-subextension} &M \ar@{-}[dd] \ar@{-}[l] \ar@{-}@/^1pc/[dd]^-{\text{unramified everywhere $A_4$-extension}}\\
\\
\text{ unique quadratic subextension} &E \ar@{-}[d] \ar@{-}[l] \ar@{-}@/^1pc/[d]^-{\text{ramified only at } p}\\
& \Q
}
\]

\noindent
Now, it suffices to show that $K=M$. Let $N=K\cap M$. Then, since $K$ and $M$ are Galois over $E$, so is $N$. Also since the normal subgroups of $\Gal(K/E)\simeq A_4 \simeq \Gal(M/E)$ are either $\{1\}, V_4$ or $A_4$, 
\[
\Gal(N/E)\simeq \text{ either } \{1 \}, \zmod 3 \text{ or } A_4.
\]
Note that the class group of $E$ is $\zmod 3$. Let $H$ be the Hilbert class field of $E$. Then, the class group of $H$ is $V_4$. 
(This can easily be checked because the degree of $H/\Q$ is small.)
If $\Gal(N/E) \simeq \{1\}$, then $E$ has two different degree $3$ unramified extensions given by $K^{V_4}$ and $M^{V_4}$, which is a contradiction. If $\Gal(N/E) \simeq \zmod 3$, then $N=H$ and $N$ has two different unramified 
$V_4$ extensions $K$ and $M$, which is a contradiction. Thus, $\Gal(N/E)\simeq A_4$ and hence $K=N=M$, as desired.
\item
This is proved in (\ref{3}).
\end{enumerate}
\end{proof}
\ms

Thus, we can take $D=2777$.
Let $F=\Q(\sqrt{-2777\cdot t})$ for a positive squarefree integer $t$ prime to $2777$. Then, $F\cap L=\Q$ because $L$ has a unique quadratic subfield $\Q(\sqrt{2777})$, which is real. 
\begin{prop} \label{nonabelian}
Let $\rho$ and $c$ be chosen so that $F^{\a}=F(\sqrt{D})$, $F^\ur = FK$ and $F^+=FL$. Then, 
\[
CS_c([\rho]) = \frac{1}{2} \Longleftrightarrow \leg {-t}{2777}=\leg {t}{2777} =-1.
\]
\end{prop}
\begin{proof}
Since $(D_L, D)=2777$ and $F^\a=F(\sqrt{D})=F(\sqrt{-t})$, the result follows from Theorem \ref{thm:final formula}.
\end{proof} 

\begin{rem}
Even in the non-abelian case, we have infinite family of non-vanishing arithmetic Chern-Simons invariants!
\end{rem}

\section{Application}\label{sec: application}
In this section, we give a simple arithmetic application of our computation. Namely, we show non-solvability of a certain case of the embedding problem based on our examples of non-vanishing arithmetic Chern-Simons invariants. 

\ms

For an odd prime $p$, let $p^*=(-1)^{\frac{p-1}{2}} p$. Let 
\[
d_1 = \prod_{i=1}^s p_i^* \qa d_2 =\prod_{j=1}^t q_j^*,
\]
where $p_i, q_j$ are distinct odd prime numbers, and $d_1, d_2 >0$.
Let 
\[
A_i:=\leg {d_2}{p_i}=\prod_{1 \leq j \leq t} \leg {q_j}{p_i} \qa B_j:=\leg {d_1}{q_j}=\prod_{1 \leq i \leq s} \leg {p_i}{q_j}.
\]
Let 
\[
\Delta(d_1, d_2):=\prod_{1 \leq i \leq s} A_i \qa \Delta(d_2, d_1):=\prod_{1 \leq j \leq t} B_j.
\]
\begin{lem}
$\Delta(d_1, d_2) =\Delta(d_2, d_1)$.
\end{lem}
\begin{proof}
Note that 
$(-1)^{\Delta(d_1, d_2)} = \prod_{\substack{1 \leq i \leq s \\ 1 \leq j \leq t}} \leg {p_i}{q_j}$.
Since $d_1$ is positive, the number of prime divisors of $d_1$, which are congruent to $3$ modulo $4$ is even. And the same is true for $d_2$. Thus by the quadratic reciprocity law,
\[
A_i = \prod_{1 \leq j \leq t} \leg {q_j}{p_i}=\prod_{1 \leq j \leq t} \leg {p_i}{q_j}.
\]
By taking product for all $1\leq i \leq s$, we get the result.
\end{proof}

\begin{prop}\label{prop:application}
Let $K=\Q(\sqrt{d_1}, \sqrt{d_2})$. If $\Delta(d_1, d_2)=-1$, then
there cannot exist a number field $L$ with odd discriminant, such that $\Gal(L/\Q) \simeq \cQ$ and $K \subset L$.
\end{prop}

A referee of an earlier version of this paper has pointed out that this result  can also be obtained using the 
theorem\footnote{$K$ extends to a quaternion extension if and only if the Hilbert symbols $(d_1, d_2)$ and $(d_1 d_2, -1)$ agree in the Brauer group.} of Witt in \cite[p. 244]{Wi} (or (7.7) on \cite[p. 106]{Fr}). (In our situation, if such a field $L$ exists, the theorem implies $\Delta(d_1, d_2)=1$, which gives us a contradiction.)  So this proposition should be viewed as a new perspective rather than a new result. In fact,  propositions \ref{prop:application} and \ref{prop:application2}  deal with a  class of embedding problems wherein the existence of an unramified extension forces a Chern-Simons invariant to be zero. The outline of proof together with the explicit formulas for computing the Chern-Simons invariant should make clear that even the simplest $\zmod 2$-valued case is likely to have a non-trivial range of applications. We consider  the point of view presented here  as  a simple and rough analogue of the classical theorem of Herbrand, whereby the existence of certain unramified extensions of cyclotomic fields forces some $L$-values to be congruent to zero (\cite[\textsection 6.3]{washington}). In future papers, we hope to discuss this analogy in greater detail and investigate the possibility of `converse Herbrand' type results in the spirit of Ribet's theorem \cite{ribet}.

\begin{proof}
Suppose that there does exist such a field $L/\Q$ satisfying all the given properties above.
Choose a prime $\ell$ such that
\begin{itemize}
\item
$\ell$ does not divide $d_1 d_2$.
\item
$\ell \equiv 3 \pmod 4$.
\item
$\leg {-\ell}{p_i} = A_i$ and $\leg {-\ell}{q_j}=B_j$ for all $i$ and $j$.
\end{itemize}
In fact, $\ell \equiv a \pmod {4d_1d_2}$ for some $a$ with $(a, 4d_1d_2)=1$, and hence there are infinitely many such primes by Dirichlet's theorem.

Now let $d_3:=\ell^*=-\ell$. And let $F=\Q(\sqrt{d_1d_2d_3})$. Then by direct computation using the quadratic reciprocity law, we get
\[
\leg {d_1d_2}{\ell}=\prod_{1\leq i \leq s} \leg {p_i}{\ell} \prod_{1 \leq j \leq t} \leg {q_j}{\ell} 
=\prod_{1 \leq i \leq s} \leg {-\ell}{p_i} \prod_{1 \leq j \leq t} \leg {-\ell}{q_j}=\Delta(d_1, d_2) \cdot \Delta(d_2, d_1).
\]
Thus by the above lemma, we get
\[
\leg {d_1 d_2}{\ell}=1.
\]
Furthermore, for all $i$ and $j$
\[
\leg {d_2 d_3}{p_i}=A_i^2=1 \qa \leg {d_3 d_1}{q_j}=B_j^2=1.
\]
Therefore by \cite[Theorem 1]{lemme}, there is a Galois extension $M/\Q$ such that $M/F$ is unramified everywhere, and
$\Gal(M/F)\simeq \cQ$. Furthermore $KF=F(\sqrt{d_1}, \sqrt{d_2})$ is the unique subfield of $M$ with $\Gal(M/{KF})\simeq \zmod 2$. 

Let $A=V_4$, and let $c_i=\a_i \cup \e$, where 
$\a_i \in H^1(A, \mu_2)$ and $\e \in Z^2(A, \zmod 2)$ represents the extension $\cQ$.
Since $M/F$ is an unramified $\cQ$-extension, $[\e]=0 \in H^2(\pi, \zmod 2)$, where
$\pi=\pi_1(\Spec(\cO_F), \fb)$ as before. Thus, $[c_i]=0 \in H^3(X, \mu_2)$ for all $i$.
This implies that $CS_{c_i}([\rho])=0$ for all $i$, where $\rho \in \Hom(\pi, A)$ factors through
\[
\pi \surj \Gal(KF/F)\simeq A.
\]
Take $\a_1$ so that $F^{\a_1}=F(\sqrt{d_1})$. Since
\[
\prod_{1 \leq i \leq s} \leg {-d_2\cdot \ell}{p_i} \times \prod_{1 \leq j \leq t} \leg {d_1}{q_j}
=\prod_{1 \leq j \leq t} B_j = \Delta(d_2, d_1)=\Delta(d_1, d_2)
=-1
\]
by assumption, we get
\[
CS_{c_1}([\rho])=\frac{1}{2}
\]
by Proposition \ref{prop:V4 general example}, which is a contradiction. Thus, there cannot exist such $L$.
\end{proof}
\ms
\begin{rem}
For the explicit construction of quaternion extensions $L$ of $\Q$, see \cite{fujisaki} or \cite[Theorem 4.5]{vaughan}.
\end{rem}

In the LMFDB, you can search for $\cQ$-extensions $L$ over $\Q$ with odd discriminants. We make a table for readers, which verifies our theorem numerically. Here $\Delta=\Delta(d_1, d_2)$.
\ms

\begin{center}
\begin{tabular}{| c | c | c | c | c | c | c | c | c | c | c | c |}
\hline
$d_1$ & $d_2$ & $\Delta$ & $\exists L$? & $d_1$ & $d_2$ & $\Delta$& $\exists L$? & $d_1$ & $d_2$ & $\Delta$& $\exists L$? \\ \hline
$5$ & $13$ & $-1$ & No & $13$ & $17$ & $1$ & Yes \cite{lmfdb21}& $17$ & $21$ & $1$ & Yes \cite{lmfdb31}\\ \hline
$5$ & $17$ & $-1$ & No & $13$ & $21$ & $-1$ & No & $17$ & $29$ & $-1$ & No \\ \hline
$5$ & $21$ & $1$ & Yes \cite{lmfdb11}& $13$ & $29$ & $1$ & Yes \cite{lmfdb22}& $17$ & $33$ & $1$ & Yes \cite{lmfdb32}\\ \hline
$5$ & $29$ & $1$ & Yes \cite{lmfdb12}& $13$ & $33$ & $-1$ & No & $17$ & $37$ & $-1$ & No \\ \hline
$5$ & $33$ & $-1$ & No & $13$ & $37$ & $-1$ & No & $17$ & $41$ & $-1$ & No \\ \hline
$5$ & $37$ & $-1$ & No & $13$ & $41$ & $-1$ & No & $17$ & $53$ & $1$ & Yes \cite{lmfdb33}\\ \hline
$5$ & $41$ & $1$ & Yes \cite{lmfdb13}& $13$ & $53$ & $1$ & Yes \cite{lmfdb23}& $17$ & $57$ & $-1$ & No \\ \hline
$5$ & $53$ & $-1$ & No & $13$ & $57$ & $-1$ & No & $17$ & $61$ & $-1$ & No \\ \hline
$5$ & $57$ & $-1$ & No & $13$ & $61$ & $1$ & Yes \cite{lmfdb24}& $17$ & $65$ & $-1$ & No \\ \hline
$5$ & $61$ & $1$ & Yes \cite{lmfdb14}& $13$ & $69$ & $1$ & No & $17$ & $69$ & $1$ & Yes \cite{lmfdb34} \\ \hline
\end{tabular}\\
\vspace{2mm}
Table 1.
\end{center}
\ms

When $d_1=13$ and $d_2=3\cdot 23=69$, there cannot exist such $L$ even though $\Delta(d_1, d_2)=1$. This follows from the following proposition which is already known to experts (e.g. \cite{vaughan}). For the sake of readers, we provide a complete proof as well.
\begin{prop}\label{prop:application2}
Let $K=\Q(\sqrt{d_1}, \sqrt{d_2})$ as above. Let $p$ be a prime divisor of $d_i$, which is congruent to $3$ modulo $4$. 
If $\leg {d_{3-i}}{p}=1$, then there cannot exist a number field $L$ such that $\Gal(L/\Q)\simeq \cQ$ and $K \subset L$.
\end{prop}
\begin{proof}
Let $p$ be a prime divisor of $d_2$, which is congruent to $3$ modulo $4$.
Suppose that $\leg {d_{1}}{p}=1$ and there exists such a field $L$.
Then by the same argument as in Proposition \ref{prop:quaternion}, the ramification index of $p$ in $L/\Q$ is $4$.  
Let $\cO=\Z[\sqrt{d_1}]$ be the ring of integers of $\Q(\sqrt{d_1})$.
Then, since $\leg {d_1}{p}=1$, $p\cO=\wp\cdot \wp'$ for two different maximal ideals $\wp$ and $\wp'$. 
Thus, $D(\wp)=I(\wp) \simeq \zmod 4$, where $D(\wp)$ (resp. $I(\wp)$) is the decomposition group (resp. inertia group) of $\wp$ in $\Gal(L/\Q)\simeq \cQ$.
Since $\cO_\wp \simeq \Z_p$, the $D(\wp)=I(\wp) \simeq \zmod 4$ can be regarded as a quotient of $\Z_p^\times \simeq \zmod {(p-1)} \times \Z_p$.
Because $p-1 \equiv 2 \pmod 4$, this is a contradiction and hence the result follows.
\end{proof}

\ms

\vspace{7mm}
\appendix

\section{Conjugation on group cochains} \label{sec:Appendix A}

We compute cohomology of a topological group $G$ with coefficients in a topological abelian group $M$ with continuous $G$-action using the complex whose component of degree $i$ is
$C^i(G, M)$, the continuous maps from $G^i$ to $M$. The differential
$$d: C^i(G,M)\to C^{i+1}(G,M)$$
is given by
$$df(g_1, g_2, \ldots, g_{i+1})$$
$$=g_1f(g_2, \ldots, g_{i+1})+
\sum_{k=1}^i f(g_1, \ldots, g_{k-1}, g_k g_{k+1}, g_{k+2}, \ldots, g_{i+1})+
(-1)^{i+1} f(g_1, g_2, \ldots, g_i).$$
We denote by
$$B^i(G,M)\subset Z^i(G,M)\subset C^i(G,M)$$
the images and the kernels of the differentials, the coboundaries and the cocycles, respectively. The cohomology is then defined as
$$H^i(G,M):=Z^i(G,M)/B^i(G,M).$$
There is a natural right action of $G$ on the cochains given by 
$$a:  c\mapsto c^a:=a^{-1}c\circ \Ad_a,$$
where $\Ad_a$ refers to the conjugation action of $a$ on $G^i$. 
\begin{lem} The $G$ action on cochains commutes with $d$:
$$d(c^a)=(dc^a)$$
for all $a\in G$.
\end{lem}

\begin{proof}
If $c\in C^i(G,M)$, then
$$d(c^a)(g_1, g_2, \ldots, g_{i+1})=g_1a^{-1}c(\Ad_a(g_2), \ldots, \Ad_a(g_{i+1}))$$
$$+
\sum_{k=1}^i a^{-1}c(\Ad_a(g_1), \ldots, \Ad_a( g_{k-1}), \Ad_a( g_k) \Ad_a( g_{k+1}), \Ad_a(g_{k+2}), \ldots, \Ad_a(g_{i+1}))$$
$$+
(-1)^{i+1} a^{-1}c(\Ad_a(g_1), \Ad_a(g_2), \ldots, \Ad_a(g_i))$$
$$=
a^{-1}\Ad_a(g_1)c(\Ad_a(g_2), \ldots, \Ad_a(g_{i+1}))$$
$$+
\sum_{k=1}^i a^{-1}c(\Ad_a(g_1), \ldots, \Ad_a( g_{k-1}), \Ad_a( g_k) \Ad_a( g_{k+1}), \Ad_a(g_{k+2}), \ldots, \Ad_a(g_{i+1}))$$
$$+
(-1)^{i+1} a^{-1}c(\Ad_a(g_1), \Ad_a(g_2), \ldots, \Ad_a(g_i))$$
$$=a^{-1}(dc)(\Ad_a(g_1), \Ad_a(g_2), \ldots, \Ad_a(g_{i+1}))$$
$$=(dc)^a(g_1, g_2, \ldots, g_{i+1}).$$
\end{proof}
We also use the notation $(g_1, g_2, \ldots, g_i)^a:=\Ad_a(g_1, g_2, \ldots, g_i)$.
It is well known that this action is trivial on cohomology. We wish to show the construction of explicit $h_a$ with the property that $$c^a=c+dh_a$$
for cocycles of degree 1, 2, and 3. The first two are relatively straightforward, but degree 3 is somewhat delicate. In degree 1, first note that
$c(e)=c(ee)=c(e)+ec(e)=c(e)+c(e)$, so that $c(e)=0$. Next, $0=c(e)=c(gg^{-1})=c(g)+gc(g^{-1})$, and hence, $c(g^{-1})=-g^{-1}c(g).$
Therefore, 
$$c(aga^{-1})=c(a)+ac(ga^{-1})=c(a)+ac(g)+agc(a^{-1})=c(a)+ac(g)-aga^{-1}c(a).$$
From this, we get
$$c^a(g)=c(g)+a^{-1}c(a)-ga^{-1}c(a).$$
That is,
$$c^a=c+dh_a$$
for the zero cochain $h_a(g)=a^{-1}c(a).$

\ms

\begin{lem}\label{lem: appendix 6.2}
For each $c\in Z^i(G,M)$ and $a\in G$, we can associate an $$h^{i-1}_a[c]\in C^{i-1}(G,M)/B^{i-1}(G,M)$$ in such a way that
\begin{equation*}
\begin{aligned}
(1) \quad & c^a-c=dh^{i-1}_a[c]; \qquad\qquad\\
(2) \quad & h_{ab}^{i-1}[c]=(h^{i-1}_a[c])^b+h^{i-1}_b[c].
\end{aligned}
\end{equation*}
\end{lem}
\begin{proof}
This is clear for $i=0$ and we have shown above the construction of $h^0_a[c]$ for $c\in Z^1(G,M)$ satisfying (1). Let us check the condition (2):
$$h^0_{ab}[c](g)=(ab)^{-1}c(ab)$$
$$=b^{-1}a^{-1}(c(a)+ac(b))=b^{-1}h^0_a[c](\Ad_b(g))+h^0_b[c](g)=(h^0_a[c])^b(g)+h^0_b[c](g).$$
We prove the statement using induction on $i$, which we now assume to be $\geq 2$.
For a module $M$, we have the exact sequence
$$0\to M\to C^1(G,M)\to N\to 0,$$
where $C^1(G,M)$ has the right regular action of $G$ and $N=C^1(G,M)/M$. Here, we give $C^1(G,M)$ the topology of pointwise convergence. There is a canonical linear splitting $s: N\to C^1(G,M)$ with image the group of functions $f$ such that $f(e)=0$, using which we topologise $N$. According to \cite[Proof of 2.5]{mostow}, the $G$-module $C^1(G,M)$ is acyclic\footnote{The notation there for $C^1(G,M)$ is $F^0_0(G,M)$. One difference is that Mostow uses the complex $E^*(G,M)$ of equivariant homogeneous cochains in the definition of cohomology. However, the  isomorphism $E^n\to C^n$ that sends $f(g_0, g_1, \ldots, g_n)$ to
$f(1, g_1, g_1g_2, \ldots, g_1g_2\cdots g_n)$ identifies the two definitions. This is the usual comparison map one uses for discrete groups, which clearly preserves continuity. }, that is, $$H^i(G, C^1(G,M))=0$$ for $i>0$. Therefore, given a cocycle $c\in Z^i(G, M)$, there is an $F\in C^{i-1}(G, C^1(G,M))$ such that
its image $f \in C^{i-1}(G,N)$ is a cocycle and $dF=c$. Hence, $d(F^a-F)=c^a-c$. Also, by induction, there is a $ k_a\in C^{i-2}(G,N)$ such that $f^a-f=dk_a$ and $k_{ab}=(k_a)^b+k_b+dl$ for some $l\in C^{i-3}(G,N)$ (zero if $i=2$). Let $K_a=s\circ k_a $ and put $$h_a=F^a-F-dK_a.$$ Then the image of $h_a$ in $N$ is zero, so $h_a$ takes values in $M$, and $dh_a=c^a-c$. Now we check property (2). Note that
$$K_{ab}=s\circ k_{ab} =s\circ (k_a)^b+s\circ k_b+s\circ dl.$$ But $s\circ (k_a)^b-(s\circ k_a)^b$ and $s\circ dl-d(s\circ l)$ both have image in $M$. Hence, $K_{ab}=K_a^b+K_b+d(s\circ l)+m$ for some cochain $m\in C^{i-2}(G,M)$. From this, we deduce $$dK_{ab}=(dK_a)^b+dK_b+dm,$$from which we get
$$h_{ab}=F^{ab}-F-dK_{ab}=(F^a)^b-F^b+F^b-F-(dK_a)^b-dK_b-dm=(h_a)^b+h_b+dm.$$
\end{proof}

\section{Conjugation action on group cochains: categorical approach}\label{sec:appendix B}
\begin{flushright}
\author{by Behrang Noohi\footnote{School of Mathematical Sciences
Queen Mary, Univ. of London,
Mile End Road London E1 4NS.}}
\end{flushright}
\ms

In this section, an alternative and conceptual proof of Lemma \ref{lem: appendix 6.2} is outlined. Although not strictly necessary for the purposes of this paper, we believe that a functorial theory of secondary classes in group cohomology will be important in future developments. This point has also been emphasised to M.K. by Lawrence Breen. More details and elaborations will follow in a forthcoming publication by B.N.

\subsection{Notation}
In what follows $G$ is a group and $M$ is a left $G$-module. The action is denoted by
$\act{a}{m}$. The left conjugation action of $a\in G$ on $G$ is denoted
$\ad_a(x)=axa^{-1}$. We have an induced right action on 
$n$-cochains  $f \: G^{n} \to M$ given by
\[f^a(\mathbf{g}):=\act{a^{-1}}{(f(\ad_{a}\mathbf{g}))}.\]
Here, $\mathbf{g} \in G^n$ is an $n$-chain, and $\ad_a\mathbf{g}$ is defined componentwise.

In what follows, $[n]$ stands for the ordered set $\{0,1,\ldots,n\}$, viewed as
a category.

\subsection{Idea}
The above action on cochains respects the differential, hence passes to cohomology. It is 
well known that the induced action on cohomology is trivial. That is, given an
$n$-cocycle $f$ and any element $a\in G$, the difference $f^a - f$ is a coboundary.
In this appendix we explain how to construct an $(n-1)$-cochain
$h_{a,f}$ such that $d(h_{a,f})=f^a- f$. The construction, presumably well known,
uses standard ideas from simplicial homotopy theory. The general case of this construction,
as well as the missing proofs of some of the statements in this appendix will appear 
in a separate article.

Let $\mcG$ denote the one-object category (in fact, groupoid)  with morphisms $G$.
For an element $a \in G$, we have an action of $a$ on $\mcG$ which, by abuse of notation,
we will denote again by $\ad_a \: \mcG \to \mcG$; it fixes the unique object and acts
on morphisms by conjugation by $a$.

The main point in the construction of the cochain $h_{a,f}$ is that there is a 
``homotopy'' (more precisely, a natural transformation) $H_a$ from  the 
identity functor $\id \: \mcG \to \mcG$ to $\ad_a \: \mcG \to \mcG$.  
The homotopy between $\id$ and $\ad_a$  is given by the functor
$H_a \: \mcG\times[1] \to \mcG$ defined by 
      \[ H_a|_0= \id, \ \ H_a|_1=\ad_a, \text{ and } H_a(\iota) = a^{-1}.\]
It is useful to visualise the category $\mcG\times[1]$ as
    \[\xymatrix@C=30pt@R=12pt@M=3pt{ 0 \ar[r]^{\iota} \ar@(ur,ul)[]_G 
                          & 1 \ar@(ur,ul)[]_G}.\]

\subsection{Cohomology of categories}
We will use multiplicative notation for morphisms in a category, namely,
the  composition of $g \: x \to y$ with $h \: y \to z$ is denoted  $gh \: x \to z$.

Let $\mcC$ be a small category and $M$ a left $\mcC$-module, that is, a functor 
$M\: \mcC^{\text{op}} \to \mathbf{Ab}$, $ x \mapsto M_x$, to the category of abelian groups
(or your favorite linear category). 
Note that when $\mcG$ is as above, this is  nothing but a left $G$-module 
in the usual sense. For an arrow $g \: x \to y$ in $\mcC$, we denote the induced map
$M_y \to M_x$ by $m \mapsto \act{g}{m}$.

Let $\mcC^{[n]}$ denote the set of all $n$-tuples $\bfg$ of composable arrows in  $\mcC$,
    \[\bfg \ = \ \bullet \xrightarrow{g_1} \bullet \xrightarrow{g_2} 
           \cdots \xrightarrow{g_n}\bullet.\]
We refer to such a $\bfg$ as an  {\bf $n$-cell} in $\mcC$; this is the same thing
as a functor $[n] \to \mcC$,
which we will denote, by abuse of notation, again by $\bfg$. 

An {\bf $n$-chain} in $\mcC$ is an element in the free abelian
group $\oC_n(\mcC,\mathbb{Z})$ generated by the set $\mcC^{[n]}$ of $n$-cells.
For an $n$-cell $\bfg$ as above, we let $s\bfg \in \Ob\mcC$ denote  the source of $g_1$.

By an {\bf $n$-cochain} on $\mcC$ with values in $M$ we mean a map
$f$ that assigns to any $n$-cell $\bfg \in \mcC^{[n]}$ an element in $M_{s\bfg}$.
Note that, by linear extension, we can evaluate $f$ on any $n$-chain 
in which all $n$-cells share a common source point.

The $n$-cochains form an abelian group  $\oC^n(\mcC,M)$.
The {\bf cohomology} groups $\oH^n(\mcC,M)$, $n\geq 0$, are defined using the  
cohomology complex $\oC^{\bullet}(\mcC,M)$: 
   \[  0 \xrightarrow{}\oC^0(\mcC,M) \xrightarrow{d}  \oC^1(\mcC,M) \xrightarrow{d}
   \cdots \xrightarrow{d} \oC^n(\mcC,M) \xrightarrow{d} \oC^{n+1}(\mcC,M) 
   \xrightarrow{d} \cdots\]
where the differential 
\[d \:\oC^n(\mcC,M) \to \oC^{n+1}(\mcC,M)\] 
is defined by
  \begin{multline*}
   df(g_1,g_2,\ldots,g_{n+1})   =  \act{g_1}(f(g_2,\ldots,g_{n+1})) +
   \underset{1\leq i \leq n}\sum(-1)^if(g_1,\ldots,g_ig_{i+1},\ldots,g_{n+1}) \\
   + (-1)^{n+1}f(g_1,g_2,\ldots,g_{n}).
  \end{multline*}

A left $G$-module $M$ in the usual sense gives rise to a left module on $\mcG$, 
which we denote again by $M$. We sometimes denote $\oC^{\bullet}(\mcG,M)$
by $\oC^{\bullet}(G,M)$. Note that the corresponding cohomology groups coincide with
the group cohomology  $\oH^n(G,M)$.

The cohomology complex $\oC^{\bullet}(\mcC,M)$ and the cohomology groups $\oH^n(\mcC,M)$
are functorial in $M$. They are also functorial in $\mcC$ in the following sense. 
A functor $\varphi \: \mcD \to \mcC$ gives rise to a $\mcD$-module  
$\varphi^*M:=M\circ \varphi \: \mcD^{op} \to \mathbf{Ab}$. We have a  map of complexes 
  \begin{equation}\label{Eq:1} \varphi^* \: \oC^{\bullet}(\mcC,M) 
     \to \oC^{\bullet}(\mcD,\varphi^*M), 
  \end{equation}
which gives rise to the maps
  \[ \varphi^* \: \oH^{n}(\mcC,M) \to \oH^{n}(\mcD,\varphi^*M) \]
on cohomology, for all $n\geq 0$.

\subsection{Definition of the cochains $h_{a,f}$}
The flexibility we gain by working with chains on general categories allows us
to import standard ideas from topology to this setting. The following definition
of the cochains $h_{a,f}$ is an imitation of a well known construction in topology.

Let $f \in \oC^{n+1}(G,M)$ be an $(n+1)$-cochain, and $a \in G$ an element. Let
$H_a \: \mcG\times[1] \to \mcG$ be the corresponding natural
transformation. We define $h_{a,f} \in \oC^{n}(G,M)$ by
      \[h_{a,f}(\bfg)=f(H_a(\bfg\times[1])).\]
Here, $\bfg \in \mcC^{[n]}$ is an $n$-cell in $\mcG$, so $\bfg\times[1]$ is an  
$(n+1)$-chain in $\mcG\times[1]$, namely, the cylinder over $\bfg$.

To be more precise, we are using the notation $\bfg\times[1]$ for the image of 
the fundamental class of $[n]\times [1]$ in $\mcG\times[1]$ under the 
functor $\bfg \times [1] \: [n]\times [1] \to \mcG \times [1]$. 
We visualize $[n]\times [1]$ as
     \[\xymatrix@C=10pt@R=12pt@M=6pt{ 
         (0,1) \ar[r]        & (1,1) \ar[r]         & \ar[r] \cdots &(n,1)\\
         (0,0) \ar[r] \ar[u] & (1,0) \ar[r] \ar[u]  & \ar[r] \cdots &
         (n,0)   \ar[u] }\]
Its fundamental class  is the alternating sum of the $(n+1)$-cells
     \[\xymatrix@C=10pt@R=12pt@M=6pt{ 
                       & &   (r,1) \ar[r]               & \cdots  \ar[r] &(n,1)\\
         (0,0) \ar[r]  & \cdots  \ar[r] &(r,0) \ar[u]   &  &   }   \]
in  $[n]\times [1]$, for $0\leq r \leq n$. Therefore, 
   \begin{equation}
     h_{a,f}(\bfg) = \underset{0\leq r \leq n}{\sum} (-1)^r
     f(g_1,\ldots,g_r,a^{-1},\ad_a{g_{r+1}},\ldots,\ad_a{g_{n}}).
   \end{equation}

The following proposition can be proved using a variant of Stokes' 
formula for cochains.

\begin{prop}{\label{P:homotopy}}
The graded map $h_{-,a}\: \oC^{\bullet+1}(G, M) \to \oC^{\bullet}(G,M)$ 
is a chain homotopy between the chain maps 
     \[\id,(-)^a \: \oC^{\bullet}(G,M) \to \oC^{\bullet}(G,M).\]
That is,
     \[ h_{a,df} + d(h_{a,f}) = f^a - f\]
for every $(n+1)$-cochain $f$. In particular, if $f$ is an $(n+1)$-cocycle, 
then  $d(h_{a,f}) = f^a - f$.  
\end{prop}

\subsection{Composing natural transformations}
Given an $(n+1)$-cochain  $f$, and elements $a,b \in G$, we can construct three 
$n$-cochains:
$h_{a,f}$, $h_{b,f}$ and $h_{ab,f}$. A natural question to ask is whether
these three cochains satisfy a cocycle condition. It turns out that the answer
is yes, but only up to a coboundary $dh_{a,b,f}$. Below we explain how 
$h_{a,b,f}$ is constructed. In fact, we construct cochains $h_{a_1,\ldots,a_k,f}$, for
any $k$ elements $a_i \in G$, $1\leq i \leq k$, and study their relationship.

Let $f \in \oC^{n+k}(G,M)$ be an $(n+k)$-cochain. Let 
$\bfa=(a_1,\ldots,a_k)\in G^{\times k}$. Consider the category $\mcG\times [k]$,
      \[\xymatrix@C=30pt@R=12pt@M=3pt{ 0 \ar[r]^{\iota_0} \ar@(ur,ul)[]_G 
          & 1 \ar@(ur,ul)[]_G    \ar[r]^(0.37){\iota_1} 
          & \ \ \cdots \ \ \ar[r]^(0.6){\iota_{k-1}} & k  \ar@(ur,ul)[]_G.}\]

Let $H_{\bfa} \: \mcG\times [k] \to \mcG$ be the functor such that 
$\iota_i \mapsto a_{k-i}^{-1}$ and $H_{\bfa}|_{\{0\}}=\id_G$. 
(So, $H_{\bfa}|_{\{k-i\}}=\ad_{a_{i+1}\cdots a_k}$.) Define
$h_{\bfa,f} \in \oC^{n}(G,M)$ by
    \begin{equation}{\label{Eq:haf}}
      h_{\bfa,f}(\bfg)=f(H_{\bfa}(\bfg\times[k])).
    \end{equation}
Here, $\bfg \in \mcC^{[n]}$ is an $n$-cell in $\mcG$, so $\bfg\times[k]$ is an  
$(n+k)$-chain in $\mcG\times[k]$.

To be more precise, we are using the notation $\bfg\times[k]$ for  the image of 
the fundamental class of $[n]\times [k]$ in $\mcG\times[k]$ under the 
functor $\bfg \times [k] \: [n]\times [k] \to \mcG \times [k]$. We visualize
$[n]\times [k]$ as
     \[\xymatrix@C=10pt@R=12pt@M=6pt{ 
         (0,k) \ar[r]        & (1,k) \ar[r]         & \ar[r] \cdots &  (n,k)        \\
           \vvdots\ar[u]     &   \vvdots \ar[u]     &               & \vvdots\ar[u] \\
         (0,1) \ar[r] \ar[u] & (1,1) \ar[r] \ar[u]  & \ar[r] \cdots &  (n,1) \ar[u] \\
         (0,0) \ar[r] \ar[u] & (1,0) \ar[r] \ar[u]  & \ar[r] \cdots &  (n,0) \ar[u] }\]
Its fundamental class is the $(n+k)$-chain
       \[\underset{P}\sum (-1)^{|P|}P,\]
where $P$ runs over (length $n+k$) paths starting from $(0,0)$ and ending in 
$(n,k)$. Note that such paths correspond to $(k,n)$ shuffles; $|P|$ stands for 
the parity of the shuffle (which is the same as the number of squares above the 
path in the $n\times k$ grid).

The most economical way to describe the relations between various $h_{\bfa,f}$ 
is in terms of the cohomology complex of the right module 
    \[\bbM^{\bullet}:=\uHom\left(\oC^{\bullet}(G,M),\oC^{\bullet}(G,M)\right).\]
Here, $\uHom$ stands for the enriched hom in the category of chain complexes, 
and the right action of $G$ on $\bbM^{\bullet}$ 
is induced from the right action $f \mapsto f^a$ of $G$ on the
$\oC^{\bullet}(G,M)$ sitting on the right. 
The differential on $\bbM^{\bullet}$ is defined by
    \[d_{\bbM^{\bullet}}(u)=(-1)^{|u|} u\circ d_{\oC^{\bullet}(G,M)}-
           d_{\oC^{\bullet}(G,M)}\circ u,\]
where $|u|$ is the degree of the homogeneous $u \in \oC^{\bullet}(G,M)$.

Note that, for every $\bfa \in G^{\times k}$, we have $h_{\bfa,f} \in \bbM^{-k}$.
This defines a $k$-cochain  on $G$ of degree $-k$ with values in $\bbM^{\bullet}$,
    \[h^{(k)}\: \bfa \ \mapsto h_{\bfa,-}, \ \bfa \in G^{\times k}.\]
We set $h^{(-1)}:=0$.  Note that $h^{(0)}$ is the element in $\bbM^{0}$
corresponding to the identity map $\id \: \oC^{\bullet}(G,M) \to \oC^{\bullet}(G,M)$.
  
The relations between various $h_{\bfa,f}$ can be packaged in a simple differential
relation. As in the case $k=0$ discussed in Proposition \ref{P:homotopy}, 
this proposition can be proved using a variant of Stokes' formula for cochains.

\begin{prop}{\label{P:relations}}
For every $k \geq -1$, we have $d_{\bbM^{\bullet}}(h^{(k+1)})=d(h^{(k)})$.
\end{prop}

In the above formula,  the term $d_{\bbM^{\bullet}}(h^{(k+1)})$ means that we apply 
$d_{\bbM^{\bullet}}$ to the values (in $\bbM^{\bullet}$) of the cochain $h^{(k+1)}$.
The differential on the right hand side of the formula is the differential of the
cohomology complex $\oC^{\bullet}(G,\bbM^{\bullet})$ of the (graded) right $G$-module  
$\bbM^{\bullet}$.
 
More explicitly, let $f\in \oC^{n+k}(G,M)$ be an $(n+k)$-cochain. Then, 
Proposition \ref{P:relations} states that, for every $\bfa \in G^{\times (k+1)}$, 
we have the following equality of $n$-cochains:
  \begin{eqnarray*}
    (-1)^{(k+1)} h_{a_1,\ldots,a_{k+1},df}-dh_{a_1,\ldots,a_{k+1},f}
       = & h_{a_2,\ldots,a_{k+1},f} + \\
      & \underset{1\leq i \leq k}\sum(-1)^ih_{a_1,\ldots,a_ia_{i+1},\ldots,a_{k+1},f} + \\
      &  (-1)^{k+1}h_{a_1,\ldots,a_{k},f}^{a_{k+1}}.
  \end{eqnarray*}

\begin{cor}{\label{C:left}} 
Let $f\in \oC^{n+k}(G,M)$ be an $(n+k)$-cocycle. Then, for every
$\bfa \in G^{\times (k+1)}$, the $n$-cochain 
    \[h_{a_2,\ldots,a_{k+1},f} 
         + \underset{1\leq i \leq k}\sum(-1)^ih_{a_1,\ldots,a_ia_{i+1},\ldots,a_{k+1},f}  
         + (-1)^{k+1}h_{a_1,\ldots,a_{k},f}^{a_{k+1}}\]
is a coboundary. In fact, it is the coboundary of $-h_{a_1,\ldots,a_{k+1},f}$.       
\end{cor}

\begin{exam}
Let us examine Corollary \ref{C:left} for small values of $k$. 
 \begin{itemize}
   \item[i)] For $k=0$, the statement is that, for every cocycle $f$,
    $f-f^a$ is a coboundary. In fact, it is the coboundary of $-h_{f,a}$.
    We have already seen this in Proposition \ref{P:homotopy}.
   \item[ii)] For $k=1$, the statement is that, for every cocycle $f$, the cochain
       \[h_{b,f}-h_{ab,f}+h_{a,f}^b \]  
    is a coboundary. In fact, it is the coboundary of $-h_{a,b,f}$.
 \end{itemize}
\end{exam}

\subsection{Explicit formula for $h_{a_1,\ldots,a_k,f}$}{\label{SS:Eplicit}}  
Let $f \: G^{\times (n+k)} \to M$ be an $(n+k)$-cochain, and 
$\bfa:=(a_1,a_2,\ldots,a_{k})  \in G^{\times k}$. Then, by (\ref{Eq:haf}),
the effect of  the $n$-cochain $h_{a_1,\ldots,a_k,f}$ on 
an $n$-tuple $\bfx:=(x_0,x_1,\ldots,x_{n-1}) \in G^{\times n}$
is given by:
     \[h_{a_1,\ldots,a_k,f}(x_0,x_1,\ldots,x_{n-1})
           =\underset{P}\sum (-1)^{|P|}f(\bfx^{P}),\]
where $\bfx^{P}$ is the $(n+k)$-tuple obtained by the following procedure.

Recall that $P$ is a path from $(0,0)$ to $(n,k)$ in the $n$ by $k$ grid. The
$l^{\text{th}}$ component $\bfx^{P}_l$ of $\bfx^{P}$ is determined by the 
$l^{\text{th}}$ segment on the path $P$. Namely, suppose that the coordinates 
of the starting point of this segment are  $(s,t)$. Then, 
       \[\bfx^{P}_l=a_{k-t}^{-1}\] 
if the segment is vertical, and 
       \[\bfx^{P}_l=(a_{k-t+1}\cdots a_k)x_s(a_{k-t+1}\cdots a_k)^{-1},\] 
if the segment is horizontal. Here, we use the convention that $a_0=1$.

The following example helps visualize $\bfx^{P}$:
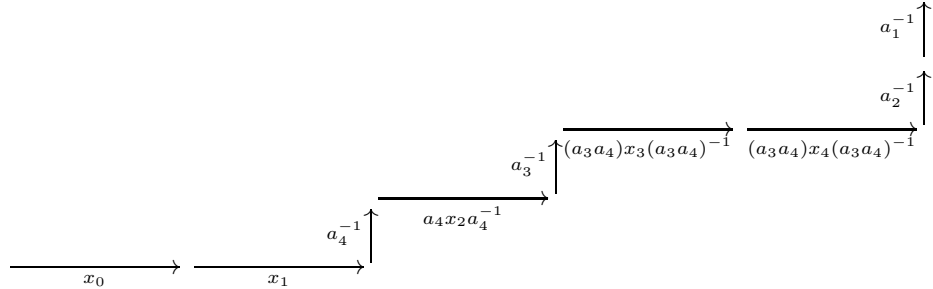
   \[\xymatrix@C=63pt@R=20pt@M=3pt{ 
              &   &  &  &  &       \\
              &   &  &  &  &  \ar[u]^{a_1^{-1}}       \\
              &   &  &  \ar[r]_{(a_3a_4)x_3(a_3a_4)^{-1}}            &  
              \ar[r]_{(a_3a_4)x_4(a_3a_4)^{-1}} & \ar[u]^{a_2^{-1}}\\
              &   &  \ar[r]_{a_4x_2a_4^{-1}}    &  \ar[u]^{a_3^{-1}} &  &\\
              \ar[r]_{x_0} & \ar[r]_{x_1}       &  \ar[u]^{a_4^{-1}} &  &  &  &}\]
The corresponding term is
     \[-f(x_0,x_1,a_4^{-1},a_4x_2a_4^{-1},
         a_3^{-1},(a_3a_4)x_3(a_3a_4)^{-1},(a_3a_4)x_4(a_3a_4)^{-1},
                                       a_2^{-1},a_1^{-1}).\]
The sign of the path is determined by the parity  of the number of squares 
in the $n$ by $k$ grid that sit above the path $P$ (in this case $15$).

\section*{Acknowledgements}
M.K. owes a tremendous debt of gratitude to many people for conversations, communications, and tutorials about a continuous stream of facts and theories that he barely understands even now. These include John Baez, Alan Barr, Bruce Bartlett, Jean Bellissard, Philip Candelas, Ted Chinburg, John Coates, Tudor Dimofte,  Dan Freed, Sergei Gukov, Jeff Harvey, Yang-Hui He, Lars Hesselholt, Mahesh Kakde, Kazuya Kato, Philip Kim, Kobi Kremnitzer, Julien March\'e, Behrang Noohi, Xenia de la Ossa,  Jaesuk Park, Alexander Schekochihin, Alexander Schmidt, Urs Schreiber, Graeme Segal,  Adam Sikora, Peter Shalen, Romyar Sharifi, Junecue Suh, Kevin Walker, and Andrew Wiles. All authors are grateful to Dr. Kwang-Seob Kim for his invaluable help in producing a number of the examples. They are also grateful to the authors of \cite{BCGKPT} for sending a preliminary version of their paper.
\ms

M.K.  was supported by EPSRC grant EP/M024830/1.

J.P. was supported by Samsung Science \& Technology Foundation (SSTF-BA1502-01).

H.Y. was supported by IBS-R003-D1.

\end{document}